\documentclass[11pt,reqno]{amsart}

\usepackage[utf8]{inputenc} 

\usepackage[margin=1in]{geometry} 

\usepackage{graphicx} 
\usepackage{float} 

 \usepackage[parfill]{parskip} 
 
\usepackage{booktabs} 
\usepackage{array} 
\usepackage{paralist} 
\usepackage{verbatim} 
\usepackage{subfig} 
\usepackage{mathrsfs}
\usepackage{amssymb}
\usepackage{amsthm}
\usepackage{amsmath,amsfonts,amssymb,esint,hyperref}
\usepackage[noabbrev, capitalize]{cleveref}
\usepackage{graphics,color}
\usepackage{enumerate, enumitem}
\usepackage{mathtools,centernot}
\usepackage{cases}
\usepackage{amsrefs}
\usepackage{bbm}
\usepackage{xfrac}

\usepackage{bookmark}

\newtheorem{theorem}{Theorem}[section]
\newtheorem{lemma}[theorem]{Lemma}

\newtheorem{corollary}[theorem]{Corollary}
\newtheorem{definition}[theorem]{Definition}
\newtheorem{proposition}[theorem]{Proposition}
\newtheorem{remark}[theorem]{Remark}

\numberwithin{equation}{section} 

\newcommand{\norm}[1]{\left\|#1\right\|}
\newcommand{\abs}[1]{\left|#1\right|}

\newcommand*{\R}{\ensuremath{\mathbb{R}}}

\newcommand*{\N}{\ensuremath{\mathbb{N}}}

\newcommand{\eps}{\varepsilon}
\newcommand*{\tr}{\ensuremath{\mathrm{tr\,}}}

\newcommand{\e}{\varepsilon}

\newcommand{\quotes}[1]{``#1''}

\renewcommand{\MR}[1]{} 

\usepackage{color, graphicx}
\usepackage{mathrsfs, dsfont}

\usepackage[]{hyperref}
\hypersetup{
    colorlinks=true,       
    linkcolor=red,          
    citecolor=blue,        
    filecolor=red,      
    urlcolor=cyan           
}

\def\dist{\mathop{\rm dist}\nolimits}    
\def\div{\mathop{\rm div}\nolimits}    

\def\spt{\mathop{\rm Spt}\nolimits} 
\def\tr{\mathop{\rm Tr}\nolimits} 
\def\Lip{\mathop{\rm Lip}\nolimits}

  \DeclareMathOperator{\trace}{tr}

\newcommand{\be}{\begin{equation}}
\newcommand{\ee}{\end{equation}}

\title{Dissipation in Onsager's critical classes and energy conservation in $BV\cap L^\infty$ with and without boundary}

\author{Luigi De Rosa}
\address{Department Mathematik Und Informatik, Universit\"at Basel, CH-4051 Basel, Switzerland}
\email{luigi.derosa@unibas.ch}

\author{Marco Inversi}
\address{Department Mathematik Und Informatik, Universit\"at Basel, CH-4051 Basel, Switzerland}
\email{marco.inversi@unibas.ch}

\date{\today}

\thanks{\textit{Acknowledgements}.
The two authors have been partially funded by the SNF grant FLUTURA: Fluids, Turbulence, Advection No. 212573. We would like to thank Camillo De Lellis for sharing the argument for the proof of \cref{T:dist function and normal} and Theodore D. Drivas for fruitful conversations. The first author is indebted to Riccardo Tione for a careful reading of a preliminary version of the current paper and for his constant interest in discussing/sharing mathematics with him, which has been extremely valuable to the author's mathematical training.
}

\begin{document}

\begin{abstract}
This paper is concerned with the incompressible Euler equations. In Onsager's critical classes we provide explicit formulas for the Duchon--Robert measure in terms of the regularization kernel and a family of vector-valued measures $\{\mu_z\}_z\subset \mathcal M_{x,t}$, having some H\"older regularity with respect to the direction $z\in B_1$. Then, we prove energy conservation for $L^\infty_{x,t}\cap L^1_t BV_x$ solutions, in both the absence or presence of a physical boundary. This result generalises the previously known case of Vortex Sheets, showing that energy conservation follows from the structure of $L^\infty\cap BV$ incompressible vector fields rather than the flow having  \textquotedblleft organized singularities". The interior energy conservation features the use of Ambrosio's anisotropic optimization of the convolution kernel and it differs from the usual energy conservation arguments by heavily relying on the incompressibility of the vector field. This is the first energy conservation proof, for a given class of solutions, which fails to simultaneously apply to both compressible and incompressible models, coherently with compressible shocks having non-trivial entropy production. To run the boundary analysis we introduce a notion of \textquotedblleft normal Lebesgue trace" for general vector fields, very reminiscent of the one for $BV$ functions. We show that having such a null normal trace is basically equivalent to have vanishing boundary energy flux. This goes beyond the previous approaches, laying down a setup which applies to every Lipschitz bounded domain. Allowing any Lipschitz boundary introduces several technicalities to the proof, with a quite geometrical/measure-theoretical flavour. 
\end{abstract}

\maketitle

\section{Introduction} Let $\Omega\subset \R^d$ be any domain with Lipschitz boundary, if any. We consider the incompressible Euler equations
\begin{equation}\label{E}
\left\{\begin{array}{l}
\partial_t u + (u \cdot \nabla)\, u +\nabla p=0 \\
\div u=0,
\end{array}\right.
\end{equation}
in $\Omega \times (0,T)$. When $\partial \Omega\neq \emptyset$, the system \eqref{E} has to be coupled with the boundary condition $u\cdot n|_{\partial \Omega}=0$, the latter to be interpreted in a suitable trace sense depending on the a priori regularity of $u$.

\subsection{A brief background} Classical computations show that regular solutions of \eqref{E} conserve the kinetic energy, that is
\begin{equation*}\label{time_der_kin_en_zero}
\frac{d}{dt} \frac12 \int_{\Omega} |u(x,t)|^2\,dx\equiv 0.
\end{equation*}
One of the possibilities to interpret the Euler system is as a model for turbulent flows in the infinite Reynolds number regime. Thus, since the celebrated work of Kolmogorov \cite{K41}, later on discussed by Onsager \cite{O49}, it has been clear that, at first approximation, a good understanding of incompressible turbulence is subject to the study of Euler solutions violating energy conservation. After the first works by Scheffer \cite{Sch93} and Shnirelman \cite{Shn00}, the construction of non-conservative (in some cases truly dissipative) solutions has been put into a rigorous mathematical framework by the convex integration methods, introduced in the context of the incompressible Euler equations by De Lellis and Székelyhidi \cites{DS13,BDSV19,Is18}. The latter, together with the rigidity part proved by Constantin, E and Titi \cite{CET94}, soon after Eyink \cite{E94}, provided an almost complete proof of the Onsager prediction about $C^{\sfrac{1}{3}}$ being the sharp threshold determining energy conservation for \eqref{E} in the class of H\"older continuous weak solutions. See also \cite{NV22} for an intermittency-based convex integration construction in Besov spaces. In view of observable turbulence, considering Besov regularity, thus not restricting to measure spatial increments in $L^\infty$, sounds more appropriate. We refer to Frisch \cite{Frisch95} for an extensive discussion on the mathematical theory of fully developed turbulence. 

We remark that the proof of the Onsager's theorem is “almost complete" since the H\"older critical case $C^{\sfrac{1}{3}}$, as well as any critical Besov, is still open. Here, the term “critical" refers to any function space which makes the energy flux bounded and not vanishing. A precise definition of energy flux can be given by means of the Duchon--Robert measure from \cite{DR00}, which will be introduced below.  Proving that a solution in a critical class conserves or dissipates energy is in general quite hard. The convex integration constructions available at the moment fail to provide Onsager's critical solutions because of infinitesimal losses along the iterations, while the energy conservation proof fails because of having non-vanishing energy flux. However, at least from a physical point of view, considering solutions having bounded energy flux is of main importance in the understanding of fully developed turbulence, since incompressible flows in the infinite Reynolds number limit inherit such property directly from Navier--Stokes. A rigorous mathematical counterpart of that has been given in \cite{DR00}, proving that anomalous Euler dissipation arise, for instance, as the limit of $D^\nu:=\nu |\nabla u^\nu|^2$ for a $L^3_{x,t}$-compact sequence of sufficiently regular vanishing viscosity solutions to the incompressible Navier-Stokes equations. The sequence $D^\nu$ stays bounded in $L^1_{x,t}$ in the natural class of solutions introduced by Leray \cite{L34}. On the opposite side, it has been also proved by the first author and Tione \cite{DT22} that energy dissipation for solutions in the supercritical class is highly unstable with respect to small perturbations. More precisely, \cite{DT22} proves that solutions of Euler having supercritical regularity generically exhibit infinite energy flux, as opposed to what physics predicts.
The emergence of this phenomenon was previously conjectured by Isett and Oh in \cite{IO17}. This gives further reasons to investigate Euler flows with critical regularity, thus not simply inheriting the finiteness of the energy flux from being vanishing viscosity limits.

To the best of our knowledge, the only known Onsager's critical class for which a complete understanding is currently available is that of classical Vortex-Sheets solutions to \eqref{E}, that is solutions (in $2$ or $3$ dimensions) whose vorticity is concentrated on a regular surface (or a curve when $d=2$). Shvydkoy \cite{Shv09} proved energy conservation for Vortex-Sheets, provided the regular evolution in time of the sheet. Indeed their energy flux is a priori non-vanishing, but the fact that they solve \eqref{E} forces the normal component of the velocity to be continuous across the sheet, together with having a continuous Bernoulli's pressure $\frac{|u|^2}{2}+p$. This is sufficient to prove that there is no energy dissipation in any space-time region. An additional folklore point of view is the following: in Onsager's critical classes, if singularities happen to be “organized", then there is no possibility for energy gaps.

In view of the aforementioned issues, this paper proves two main results. The first is to provide a formula for the energy flux of Onsager's critical vector fields, not necessarily solving the incompressible Euler equations, in which the convolution kernel used to define the flux appears explicitly. In our second result, by optimizing the choice of the kernel, we prove local interior energy conservation for $L^1_t BV_x\cap L^\infty_{x,t}$ solutions to \eqref{E}. Moreover, we analyse the role of physical boundaries, introducing a notion a normal boundary trace and studying how it is linked to the conservation of the total kinetic energy. The precise statements are given here below.

\subsection{Main results}

Recall from Duchon and Robert \cite{DR00} that, whenever $(u,p)\in L^3_{\rm loc}\times L^{\sfrac{3}{2}}_{\rm loc}$, there exists a space-time distribution $D[u]\in \mathcal D'_{x,t}$ such that 
\begin{equation}
    \label{local_energy_eq}
\partial_t \left(\frac{|u|^2}{2}\right) +\div  \left(\left(\frac{|u|^2}{2}+p\right)u\right)=-D[u] \qquad \text{in } \mathcal D'_{x,t},
\end{equation}
where
$$
D[u]=\lim_{\eps \rightarrow 0^+} D_\eps[u] \qquad \text{in } \mathcal D'_{x,t},
$$
with
\begin{equation}\label{DR_eps_formula}
D_\eps [u](x,t):= \frac{1}{4\eps} \int \nabla \rho (z) \cdot \delta_{\eps z} u(x,t)  \left| \delta_{\eps z} u(x,t) \right|^2 \,dz,
\end{equation}
and where $\delta_{\xi} u (x,t):= u(x+\xi,t)-u(x,t)$, for any $\xi\in \R^d$. In \eqref{DR_eps_formula} every 
\begin{equation}\label{kernel_in_B1}
\rho\in \mathcal{K}:=\left\{\rho \in C^\infty_c(B_1)\,:\, \int \rho(z)\, dz =1, \, \rho \geq 0, \, \rho \text{ even} \right\}
\end{equation}
is allowed, but nonetheless the limiting $D[u]$ does not depend on the choice of $\rho$. Note that when $\Omega$ is a bounded open set, \eqref{DR_eps_formula} makes sense if $x\in O\subset\joinrel\subset \Omega$ and $\eps<\dist (\overline O,\partial \Omega)$. The latter restriction is enough to define the distributional limit of $D_\eps [u]$, since we only need
$$
\lim_{\eps \rightarrow 0^+} \langle D_\eps [u], \varphi \rangle = \langle D [u], \varphi \rangle \qquad \forall \varphi \in C^\infty_c(\Omega \times (0,T)).
$$
It makes sense to define $D_\eps [u]$ for every vector field $u$, thus not necessarily restricting to solutions of Euler.
In this setting, the local energy flux of $u$ can be defined as 
\begin{equation}\label{energy_flux}
\lim_{\eps\rightarrow 0^+} \int_{O\times I} \left| D_\eps [u] \right|\,dx \, dt\qquad O\times I\subset\joinrel\subset \Omega\times (0,T).
\end{equation}
It is then natural to define Onsager's critical class any space $X$ in which $D_\eps[u]$ stays bounded in $L^1_{\rm loc}$ if $u\in X$, without having the limit in \eqref{energy_flux} equal to zero. Note that having $D_\eps [u]$ bounded in $L^1_{\rm loc}$ guarantees, possibly passing to subsequences, that $D[u]$ belongs to the space of (signed) Radon measures, the latter denoted by $\mathcal M_{\rm loc}$. In view of the mathematical theory of fully developed turbulence, the most natural critical class is $L^3_t B^{\sfrac{1}{3}}_{3,\infty}$ (see \cref{S:tools} for a precise definition), but more general spaces can be considered. Indeed, requiring more than $\sfrac{1}{3}$ of a derivative in $L^3$ forces the limit \eqref{energy_flux} to be zero \cites{DR00,CET94}. See also \cite{CCFS08} for sharper, still subcritical, results.

Our first result provides a formula for the limit measure $D[u]$ when $u$ has critical regularity, in which the dependence on the kernel $\rho$ is explicit. For $p\in [1,\infty]$ we will denote by $p'\in [1,\infty]$ its conjugate exponent, by $B_1$ the unit ball centered at the origin and by $\|\mu\|_{\mathcal M(K)}$ the total variation of a measure $\mu$ over the set $K$ (see \cref{S:tools} for the precise definition).

\begin{theorem}[Onsager's critical measure] \label{T:ons_crit}
Let $u:\Omega\times (0,T)\rightarrow \R^d$ be a vector field such that 
\begin{align}\label{Besov_assumpt}
    u\in L^p(I;B^{\sfrac{1}{p}}_{p,\infty}(O))\cap L^{2p'}(I;B^{\sfrac{1}{2p'}}_{2p',\infty}(O))\qquad \text{for some } p\in(1,\infty),
\end{align}
or 
\begin{equation}\label{1_2_Assumpt}
u\in L^2(I;B^{\sfrac{1}{2}}_{2,\infty}(O))\cap L^{\infty}(I\times O),
\end{equation}
or
\begin{equation}\label{BV_assumpt}
u\in L^1(I;BV(O))\cap L^{\infty}(I\times O),
\end{equation}
for all $I\subset\joinrel\subset (0,T)$ and $O\subset\joinrel\subset \Omega$ open sets. Let $D[u]$ be any weak limit, in the sense of measures, of the sequence $D_\eps[u]$ defined in \eqref{DR_eps_formula}, with $\rho\in \mathcal K$. There exists a map $\mu \in C^{0}(\overline B_1; \mathcal{M}_{\rm loc}(\Omega \times (0,T); \R^d ) ) $ independent of $\rho$ such that
\begin{equation} \label{eq: D_u formula}
    \langle D[u],\varphi \rangle = \frac14 \int_{B_1} \nabla \rho (z) \cdot \langle \mu_z,\varphi \rangle \,dz \qquad \forall \varphi \in C_c(\Omega \times (0,T)). 
\end{equation}
In particular, if $D[u]$ does not depend on the convolution kernel, we have that 
\begin{equation} \label{DRI_formula}
    \abs{\langle D[u], \varphi \rangle} = \frac{1}{4}\inf_{\rho \in \mathcal{K}} \abs{\int_{B_1} \nabla \rho(z)\cdot \langle \mu_z,\varphi \rangle \, dz }.
\end{equation}
Moreover, $\mu$ enjoys the following properties.
\begin{itemize}
    \item[(i)] H\"older continuity: for any compact set $K\subset \Omega\times (0,T)$ there exists a constant $C=C(u,K)>0$ such that 
    \begin{align}
    \eqref{Besov_assumpt}\, \Longrightarrow \,\norm{\mu_{z_1} - \mu_{z_2}}_{\mathcal{M}(K)} &\leq C \abs{z_1-z_2}^{\alpha_p}\left(|z_1|^{\beta_p}+  |z_2|^{\beta_p}\right), \label{eq: mu-2}\\
    \eqref{1_2_Assumpt}\, \Longrightarrow \,\norm{\mu_{z_1} - \mu_{z_2}}_{\mathcal{M}(K)} &\leq C \abs{z_1-z_2}^\frac{1}{2}\left(|z_1|^\frac{1}{2}+|z_2|^\frac{1}{2} \right), \label{eq: mu-4}\\
        \eqref{BV_assumpt}\, \Longrightarrow \,\norm{\mu_{z_1} - \mu_{z_2}}_{\mathcal{M}(K)} &\leq C \abs{z_1-z_2},\label{eq: mu-3}
\end{align}
for all $z_1,z_2\in \overline B_1$, where 
$$
\alpha_p:=\max\left(\frac{1}{p},\frac{1}{2p'} \right)\qquad \text{and}\qquad \beta_p:=\min\left(\frac{1}{p'},\frac{p+1}{2p}\right);
$$
    \item[(ii)]  Vanishing at the origin: $ \mu_0 \equiv 0 $ as an element of $\mathcal{M}_{\rm loc} (\Omega\times (0,T);\R^d)$;
    \item[(iii)]  Odd: for any $\rho\in \mathcal K$ it holds
    $$
    \int_{B_1}\nabla \rho (z)\cdot \langle \mu_{-z},\varphi\rangle \,dz=-\int_{B_1}\nabla \rho (z)\cdot \langle \mu_{z},\varphi\rangle \,dz  \qquad \forall \varphi \in C_c(\Omega \times (0,T)). 
    $$
\end{itemize}
\end{theorem}
The critical spaces \eqref{Besov_assumpt}--\eqref{BV_assumpt} identified in the above theorem are just particular choices, perhaps in view of relevant regularities in the mathematical theory of fully developed turbulence and structure function exponents, but more general critical classes exist. See for instance the recent works \cites{Ber23,BG23} which identify sharper relations between time integrability and spatial fractional regularity.

We prove \cref{T:ons_crit} in \cref{S:DR_formula}, where we also present some results on the support of the dissipation measure as well as the case of vector fields with “bounded deformation". It will be clear from the proof that the constant $C$ in the above statement can be written explicitly in terms of the corresponding local norms of $u$. We remark that $p=3$ is the critical value that determines the expression of $\alpha_p$ and $\beta_p$, that is H\"older regularity of the measures $\mu_z$ in the direction $z$. Indeed
$$
\alpha_p=\left\{\begin{array}{l}
 \frac{1}{p}\quad \text{if } p\leq 3\\
\frac{1}{2p'}\quad \text{if } p> 3
\end{array}\right.
\qquad \text{and} \qquad 
\beta_p=\left\{\begin{array}{l}
 \frac{1}{p'}\quad \text{if } p\leq 3\\
\frac{p+1}{2p}\quad \text{if } p> 3.
\end{array}\right.
$$

The proof of \cref{T:ons_crit} is based on an Ascoli--Arzel\'a  argument on the cubic expression appearing in the sequence $D_\eps [u]$. Indeed, we show that having any critical Besov regularity guarantees equi-continuity with respect to $z$ of the operator in \eqref{DR_eps_formula}, being trilinear with respect to the displacements $\delta_{\eps z} u$. This allows us to define the family of vector-valued measures $\mu_z$, having H\"older continuity with respect to $z$. We emphasize that we are not able to prove that the measure $\mu_z$ is  unique, not even if $u$ solves Euler, since it is obtained by compactness and up to subsequences. Further formulas for specific choices of $\rho$ are discussed in \cref{S:DR_formula}.

Having an explicit dependence on $\rho$ makes possible to optimize its choice whenever $D[u]$ does not depend on the kernel, that is, for instance, the case in which $u$ solves $\eqref{E}$. Then, it is natural to wonder whether the infimum in \eqref{DRI_formula} can be shown to be zero in some Onsager's critical class. We prove that this is the case for incompressible vector fields $u\in L^1_t BV_x\cap L^\infty_{x,t}$.

\begin{theorem}\label{T:cons_local_BV}
    Let $u:\Omega\times (0,T)\rightarrow \R^d$ be a divergence-free vector field such that
\begin{equation}\label{assumpt_Linfty_BV}
u\in L^1(I;BV(O))\cap L^{\infty}(I\times O),
\end{equation}
for all $I\subset\joinrel\subset (0,T)$ and $O\subset\joinrel\subset \Omega$ open sets. Let $D[u]$ be any weak limit, in the sense of measures, of the sequence $D_\eps [u]$. If $D[u]$ is independent on the choice of $\rho$, then
$$
|D[u]|(K)=0\qquad \forall K\subset \Omega\times (0,T) \text{ compact}.
$$
\end{theorem}
In the statement above we used the notation $|\mu|$ to denote the variation of a signed Radon measure $\mu$ (see \cref{S:tools}). Since from \cite{DR00} we know that the sequence $D_\eps [u]$ has a unique limit which does not depend on the kernel $\rho$ when $u$ solves the incompressible Euler system, we deduce local energy conservation for solutions to \eqref{E} in the class \eqref{assumpt_Linfty_BV}. 

The proof of \cref{T:cons_local_BV} is based on Ambrosio's anisotropic optimization of the kernel that has been introduced in \cite{Ambr04} to prove the renormalization property for weak solutions to the transport equations with $BV$ vector fields, 
and it heavily relies on the fact that $u$ is divergence-free. In particular, the same approach would fail, and it has to do so, in the compressible context. For instance, Burgers shocks 
 enjoy the same critical regularity $L^\infty\cap BV$, and they enjoy an analogous version of the formulas given in \cref{T:ons_crit}, while having non-trivial dissipation measure. Thus, the approach we use here is quite different than the usual ones, since they all simultaneously apply to both compressible and incompressible models, with the exception of \cite{Shv09} at the price of imposing geometric constraints on singularities. More precisely, as opposite as it is usually done in the context of fluid equations, \cref{T:cons_local_BV} does not follow from the strong $L^1_{\rm loc}$ convergence of the approximating energy fluxes, that would be $\lim_{\eps \rightarrow 0^+} \|D_\eps [u]\|_{L^1_{\rm loc}}=0$, but instead its proof shows that any weak limit (in the sense of measures) of $D_\eps [u]$ can be made arbitrarily small by optimizing on the choice of the kernel $\rho$. The incompressiblity of $u$ plays a key role in the optimization. Indeed, following Ambrosio's argument, we prove that
\begin{equation}\label{ambrosio_structure}
|D[u]|\ll \inf_{\rho\in \mathcal K}   \left( \int_{B_1} |\nabla \rho (z)\cdot M_{x,t} z|\,dz\right) \nu,
\end{equation}
for a positive measure $\nu$ and a matrix $M_{x,t}$ such that $\tr M_{x,t}=0$ for $\nu$-a.e. $(x,t)\in \Omega\times (0,T)$. The measure $\nu$ is basically given by the variation of the gradient of $u$, while the matrix $M_{x,t}$ is the polar part of the Radon-Nikodym decomposition of $\nabla u$ with respect to $|\nabla u|$, which then has to be trace-free by the constraint $\div u=0$. Then, Alberti's lemma (see for instance \cite{Cr09}*{Lemma 2.13}) concludes the proof by showing that 
$$
\inf_{\rho\in \mathcal K}    \int_{B_1} |\nabla \rho (z)\cdot M z|\,dz=|\tr M|.
$$
Strictly speaking, here we use a variation of Alberti's \cref{L:alberti}. Indeed, instead of $Mz$, we consider any odd (or even), Lipschitz continuous, divergence-free vector field. This shows that having the linear structure $Mz$ is not necessary to achieve zero in the infimum, but any incompressible vector field would suffice. We believe this generalization to be useful to investigate further Onsager's critical classes in which the increments $\delta_z u$, when measured in an integral sense, may not behave linearly in $z$ as it happens in the $BV$ case, see \eqref{BV_est_increment}.

Although there are different approximations for $D[u]$, the one proposed by Duchon and Robert \eqref{DR_eps_formula} seems to be necessary in order to achieve the structure \eqref{ambrosio_structure}, as opposite to the one following the Constantin, E and Titi approach \cite{CET94}. Details are given in \cref{R:DR_vs_CET}.

It is important to emphasise that, even if \cref{T:cons_local_BV} does not require any a priori \textquotedblleft organization of singularities", $BV$ functions roughly behave, at small scales, like codimension-$1$ singularities on rectifiable sets. This is a consequence of the celebrated Rank-one theorem by Alberti \cite{Al93}, which shows that the polar matrix of the singular part of the distributional gradient $\nabla u$ has rank $1$ almost everywhere with respect to its variation. However, both the fact that rectifiable sets might not look \textquotedblleft organized" at large scales and the matrix-valued measure $\nabla u$ having in general a non-trivial Cantor part (see for instance \cite{AFP00}*{Section 3.9}), make the above intuition a bit looser, at least globally.

Having established interior local energy conservation, we also address the case in which a physical boundary $\partial \Omega$ is present.  In this context we define the \emph{total} kinetic energy by 
$$
e_u(t):=\frac{1}{2}\int_{\Omega}|u(x,t)|^2\,dx.
$$
As previously noted in \cites{BTW19,DN18}, if $D[u]\equiv 0$ in $\mathcal{D}'_{x,t}$, then by \eqref{local_energy_eq} kinetic energy conservation, for instance in the sense of \cref{D:kinetic energy cons}, follows as soon as the quantity
$$
\int_0^T \int_\Omega \left|\left(\frac{|u|^2}{2}+p \right) u\cdot \nabla \varphi \right| \,dx \,dt
$$
vanishes when we let $\varphi\in C^\infty_c(\Omega)$ converge to the indicator function of $\Omega$. Assuming that $P:=\frac{|u|^2}{2}+p \in L^\infty_{x,t}$, this amounts to 
\begin{equation}\label{vanish_boundary_normal_flux_intro}
\liminf_{k\rightarrow \infty}\int_0^T\int_\Omega |u\cdot \nabla \varphi_k|\,dx\,dt=0, \qquad \text{if } \varphi_k\rightarrow \mathds{1}_{\Omega}.
\end{equation}
As emphasised in \cites{BTW19,DN18}, an assumption  on the normal component of $u$ when approaching $\partial\Omega$ is enough to have \eqref{vanish_boundary_normal_flux_intro}. For instance, the continuity of the normal velocity in a neighbourhood of the boundary has been shown to be sufficient, consistently with the possible formation of a Prandtl-type boundary layer in the vanishing viscosity limit. To this end, we give a notion of “normal Lebesgue trace" on $\partial \Omega$ (see \cref{D:Leb normal trace}), which we will denote by $u^{\partial \Omega}_n$. It turns out, see \cref{P:normal trace absol converg}, that having such a notion of normal boundary trace vanishing on a portion of $\partial\Omega$ of $\mathcal H^{d-1}$-full measure is basically equivalent to \eqref{vanish_boundary_normal_flux_intro}. Differently from the previous works \cites{BTW19,DN18} where a $C^2$ assumption on $\partial\Omega$ was required, our analysis applies to any domain with Lipschitz boundary. In particular, when $\partial \Omega$ is piece-wise $C^2$, our approach shows that no assumption on the behaviour of the normal component of $u$ around boundary corner points is required, since they would form a $\mathcal{H}^{d-1}$-negligible subset of $\partial \Omega$.

For any set $A\subset \R^d$ we will denote by $(A)_\eps$ its $\eps$ tubular neighbourhood (see the formula \eqref{eps_tub_neig_def}). 
\begin{theorem}[Energy conservation on bounded domains]\label{T:energy cons bounded dom intro}
Let $u$ be a solution to \eqref{E} on a Lipschitz bounded domain $\Omega\subset \R^d$ such that 
\begin{itemize}
    \item[(i)] $u\in L^1(I;BV (O))\cap L^\infty(I;L^\infty(\Omega))$ and $p\in L^1(O\times I)$, for all $I\subset \joinrel\subset (0,T)$ and $O\subset\joinrel\subset \Omega$ open;
    \item[(ii)] for all $I\subset\joinrel \subset (0,T)$ there exists $\eps_0>0$ such that $p\in L^1(I;L^\infty((\partial \Omega)_{\eps_0}\cap \Omega))$;
    \item[(iii)]  for almost every $t\in (0,T)$, $u(\cdot,t)$ has zero Lebesgue normal boundary trace according to \cref{D:Leb normal trace}, i.e. $u(\cdot,t)^{\partial\Omega}_n\equiv 0$.
\end{itemize}
Then $u$ conserves the total kinetic energy in the sense of \cref{D:kinetic energy cons}.
\end{theorem}

Some more effective conditions which imply $u^{\partial\Omega}_n\equiv 0$ are given in \cref{P:normal trace is zero}. We emphasise that, differently from \cites{BTW19,DN18}, our theorem provides kinetic energy conservation in an Onsager's critical class. Working on general Lipschitz domains introduces several technicalities in the proof, such as the study of the Minkowski content of closed rectifiable sets, the regularity of the distance function from a Lipschitz domain and covering arguments. We collect them in \cref{S:tools}. To conclude, in \cref{T:en cons bounded domain more general} we also state a result in which the Bernoulli's pressure $P=\frac{|u|^2}{2}+p$ is not necessarily bounded in a neighbourhood of the boundary. Thus $P$ is allowed to blow-up approaching $\partial\Omega$, provided that the normal Lebesgue boundary trace is achieved in a quantitative (in terms on how $P$ blows-up) fast enough way.

\subsection{Organization of the paper} In \cref{S:tools} we collect the technical results needed in the paper, such as definitions and basic properties of Besov, $BV$ and $BD$ spaces, properties of the distance functions from Lipschitz sets, Minkowski content of closed rectifiable sets and Ambrosio--Alberti types lemmas. \cref{S:DR_formula} is devoted to the proof, and the analysis of some consequences, of \cref{T:ons_crit}. In \cref{S:BV_cons} we discuss the proof of \cref{T:cons_local_BV}. Finally, in \cref{S:boundary} we introduce the notion of “normal Lebesgue trace" and we prove \cref{T:energy cons bounded dom intro}, together with its generalized version \cref{T:en cons bounded domain more general}. In \cref{P:normal trace is zero} we discuss some effective conditions to have vanishing normal Lebesgue boundary trace, and thus energy conservation on bounded domains.

\section{Notations and technical tools}\label{S:tools}

In this section, we recall the main definitions and list the tools used throughout the manuscript. Here $\Omega$ is any open set in $\R^d$. We will specify when we will restrict to consider bounded domains only. Moreover, we use the standard notation $v\cdot w$ to denote the Euclidean scalar product between the two vectors $v$ and $w$.

\subsection{Radon measures}  We denote by $\mathcal{M}_{\rm loc}(\Omega;\R^m)$ the space of Radon measures with values in $\R^m$, that is Borel measures, finite on compact subsets of $\Omega$. When $m =1$, we use the shorter notation $\mathcal{M}_{\rm loc}(\Omega)$. Moreover, for a vector-valued measure $\mu$ we denote by $|\mu|\in \mathcal{M}_{\rm loc}(\Omega)$ its variation, namely the positive (scalar-valued) measure defined as
$$
\langle |\mu|,\varphi\rangle:=\int_{\Omega} \varphi \,d|\mu|:= \sup_{\Phi \in C^0(\Omega;\R^m),|\Phi|\le \varphi}\int_\Omega \Phi\cdot \,d\mu\qquad \forall \varphi\in C^0_c(\Omega), \, \varphi \geq 0.
$$
We denote by $\mathcal{M}(\Omega;\R^m)$ the space of finite vector-valued Radon measures on $\Omega$, i.e. Radon vector valued-measures on $\Omega$ such that $|\mu|(\Omega)<\infty$. $\mathcal{M}(\Omega; \R^m)$ is a Banach space, with the norm 
 $$\|\mu\|_{\mathcal{M}(\Omega)}:= |\mu|(\Omega).
 $$
Then, the weak convergence in $\mathcal{M}_{\rm loc}(\Omega;\R^m)$ of $\mu_k$ to $\mu $ is given by
\[
\mu_k \rightharpoonup \mu  \quad \Leftrightarrow\quad \langle \mu_k,\Phi\rangle \to \langle \mu,\Phi\rangle, \, \, \forall \Phi \in C^0_c(\Omega;\R^m).
\]
Since $\mathcal{M}_{\rm loc}(\Omega;\R^m)$ is the dual of $C^0_c(\Omega;\R^m)$ that is a separable space, we have sequential weak-star compactness for equi-bounded sequences. See for instance \cite{EG15}.

We denote by $\mu\llcorner C$ the restriction of the measure $\mu$ to the set $C$, that is $\mu\llcorner C (A):=\mu(A\cap C)$ for every Borel set $A$.

\subsection{Besov, bounded variation and bounded deformation spaces}
For $\theta\in (0,1)$ and $p\in [1,\infty)$ we denote by $B^\theta_{p,\infty}(\R^d)$ the space of $L^p(\R^d)$ functions, with 
\begin{equation}\label{besov_norm}
[f]_{B^\theta_{p,\infty}(\R^d)}:=\sup_{|h|>0} \frac{\|f (\cdot + h)-f(\cdot)\|_{L^p(\R^d)}}{|h|^\theta}<\infty.
\end{equation}
The full norm is then given by 
$$
\|f\|_{B^\theta_{p,\infty} (\R^d)}:=\|f\|_{L^p(\R^d)}+ [f]_{B^\theta_{p,\infty}(\R^d)}.
$$

With a slight abuse of notation, we define Besov functions on a bounded domain $\Omega$ as follows: for any $O\subset \joinrel \subset \Omega$, define $B^\theta_{p,\infty}(O)$ by replacing $\R^d$  in \eqref{besov_norm} with $O$, and compute the supremum all over $|h|\leq \dist(\overline O,\partial \Omega)$, so that $f(x+h)$ is well defined. For $p=\infty$ we identify $B^\theta_{\infty,\infty}(\Omega)$ with $C^\theta(\Omega)$, the space of H\"older continuous functions. 

We denote by $BV(\Omega)$ the space of functions with bounded variations, that is functions $f$ whose distributional gradient is a finite Radon measure on $\Omega$, i.e.
$$
BV(\Omega):=\left\{f\in L^1(\Omega)\,:\, \nabla f\in \mathcal{M}(\Omega;\R^d) \right\}.
$$
The latter definition generalises to vector fields $f:\Omega\rightarrow \R^d$. The space of functions with bounded deformation is given by vector fields whose symmetric gradient is a finite measure, i.e.
$$
BD(\Omega;\R^d):=\left\{f\in L^1(\Omega;\R^d)\,:\, \frac{\nabla f+\nabla f^T}{2}\in \mathcal{M}(\Omega;\R^{d\times d}) \right\}.
$$
Throughout the paper, we will denote the symmetric part of the gradient by 
$$
\nabla^s f:=\frac{\nabla f + \nabla f^T}{2}. 
$$

Recall from \cite{AFP00}*{Theorem 3.87} that functions in $BV(\Omega)$ admit a notion of trace, on any domain $\Omega$ with Lipschitz boundary.
\begin{theorem}[Boundary trace]\label{T:trace_in_BV}
Let $\Omega\subset \R^d$ be an open set with bounded Lipschitz boundary and $f\in BV(\Omega;\R^m)$. There exists a function $f^\Omega\in L^1(\partial \Omega;\mathcal H^{d-1})$ such that 
$$
 \lim_{r\rightarrow 0^+}\frac{1}{r^d} \int_{B_r (x)\cap\Omega} \left|f(y)- f^\Omega (x)\right|\,dy=0\qquad \text{for } \mathcal{H}^{d-1}\text{-a.e. } x\in \partial \Omega.
$$
Moreover, the extension $\tilde f$ of $f$ to zero outside $\Omega$ belongs to $BV(\R^d;\R^m)$ and
$$
\nabla \tilde f=(\nabla f)\llcorner\Omega + (f^\Omega\otimes n) \mathcal{H}^{d-1} \llcorner\partial \Omega,
$$
being $n:\partial \Omega\rightarrow \mathbb{S}^{d-1}$ the inward unit normal.
\end{theorem}

In the lemma below we make use of the notation 
\begin{equation}\label{eps_tub_neig_def}
(A)_\eps:=A+B_\eps(0)
\end{equation}
for the $\eps$ tubular neighbourhood of a Borel set $A \subset\joinrel\subset \Omega$, with the implicit restriction $2\eps<\dist ( A, \partial \Omega)$ whenever $\Omega$ has a boundary. 

\begin{lemma}[$BV$ and $BD$ increment estimates]\label{L:BD_gradient}
Let $\Omega \subset \R^d$, $d\geq 2$. Suppose that $f:\Omega\rightarrow \R^d$ belongs to $BD_{\rm loc}(\Omega;\R^d)$. For any $A\subset\joinrel\subset \Omega$ Borel set, and any  $z\in B_1$, we have
\begin{equation}\label{BD_est_increment}
\int_A | z\cdot (f(x+\eps z)-f(x))|\,dx\leq \eps |z|^2 |\nabla^s f|\left( \overline{(A)_\eps}\right).
\end{equation}
Moreover, if $f\in BV_{\rm loc}(\Omega;\R^m)$, $m\geq 1$,  we have
\begin{equation}\label{BV_est_increment}
\int_A | f(x+\eps z)-f(x)|\,dx\leq \eps |z\cdot\nabla f|\left( \overline{(A)_\eps}\right),
\end{equation}
where $z\cdot \nabla f$ is the $m$-dimensional measure, whose $m$-th component is given by $(z\cdot \nabla f)_m:=z_i \partial_i f_m$, summing over repeated indexes as usual.
\end{lemma}
\begin{proof}
We start by proving the validity of \eqref{BD_est_increment} and \eqref{BV_est_increment} for smooth functions. The general case follows by density.

Let $f:\Omega\rightarrow \R^d$ be smooth, $A\subset\joinrel\subset \Omega$ and $2\eps<\dist ( A, \partial \Omega)$. For any $x\in A$ and $z\in B_1$ we have
\begin{align*}
    z_i(f_i(x+\eps z)-f_i(x))&=z_i\int_0^1 \frac{d}{dt}f_i(x+t\eps z)\,dt=\eps z_iz_j\int_0^1 (\partial_j f_i)(x+t\eps z) \,dt\\
    &=\eps \int_0^1 z\otimes z : (\nabla f)(x+t\eps z)\,dt=\eps \int_0^1 z\otimes z :(\nabla^s f)(x+t\eps z)\,dt.
\end{align*}
Thus
$$
\int_A | z\cdot (f(x+\eps z)-f(x))|\,dx\leq \eps |z|^2 \int_0^1 \int_A |\nabla^s f|(x+t\eps z)\,dx\,dt\leq \eps |z|^2 |\nabla^s f|\left( (A)_\eps\right).
$$
If $f\in BD_{\rm loc}(\Omega)$, we consider its mollification $f_\delta:=f*\rho_\delta$, for any $2\delta<\dist ( A,\partial\Omega)$. Since $f_\delta$ is smooth, by writing \eqref{BD_est_increment} for $f_\delta$ we have
$$
\int_A | z\cdot (f_\delta(x+\eps z)-f_\delta(x))|\,dx\leq  \eps |z|^2 |\nabla^s f_\delta|\left( (A)_\eps\right)\leq \eps |z|^2|\nabla^s f|\left( ((A)_\eps)_\delta\right).
$$
Since $f_\delta\rightarrow f$ strongly in $L^1_{\rm loc}(\Omega)$ and $\bigcap_{\delta>0} ((A)_\eps)_\delta=\overline{(A)_\eps}$, the latter inequality implies \eqref{BD_est_increment} for $f$, 

The proof of \eqref{BV_est_increment} follows by very similar considerations, by first showing that 
$$
\int_A | f(x+\eps z)-f(x)|\,dx\leq \eps |z\cdot\nabla f|\left( (A)_\eps\right),
$$
whenever $f$ is smooth, and then by considering the mollification $f_\delta$ in the general case.
\end{proof}

\subsection{Lipschitz sets and properties of the distance function}

For a bounded open Lipschitz set $\Omega \subset \R^d$, we denote by $d_{\partial\Omega}:\Omega\rightarrow [0,\infty)$ the interior distance to the boundary, i.e.  
$$
d_{\partial \Omega}(x):=\dist(x,\partial \Omega).
$$
The distance function $d_{\partial \Omega}$ satisfies the properties listed in the following classical lemma, which is proved for the reader's convenience.

\begin{lemma} [Distance function]\label{L:dist funct}
The function $d_{\partial \Omega}$ is Lipschitz, thus differentiable a.e. in $\Omega$, with $|\nabla d_{\partial \Omega} (x)|=1$ for a.e. $x\in \Omega$. Moreover, for almost every $x\in \Omega$, if $y\in \partial \Omega$ is any point such that $d_{\partial \Omega} (x)=|x-y|$, then 
\begin{equation}\label{gradient dist_expression}
     \nabla d_{\partial \Omega} (x)=\frac{x-y}{|x-y|}.
\end{equation}
\end{lemma}
It will be clear from the proof that \cref{L:dist funct} holds whenever one considers $d_C (x)=\dist (x,C)$, where $C\subset \R^d$ is a closed set.
\begin{proof}
    From the straightforward inequality 
    \begin{equation}
        \label{dist_is_lip}
        |d_{\partial \Omega} (x)-d_{\partial \Omega} (y)|\leq |x-y|,
    \end{equation}
    we deduce that $d_{\partial \Omega} \in \Lip (\Omega)$. Thus by Rademacher's theorem $d_{\partial \Omega}$ is differentiable a.e. in $\Omega$, with $|\nabla d_{\partial \Omega} |\leq 1$. Let $x\in \Omega$ be any point of differentiability. Since $\partial \Omega$ is closed and non-empty, we can find $y\in \partial \Omega$ such that $d_{\partial \Omega} (x)=|x-y|$. Define $f(t):=d_{\partial \Omega} (x+t(y-x))$, for $t\in [0,1]$. By \eqref{dist_is_lip} we have $f\in \Lip ([0,1])$ with $|f'(t)|\leq |y-x|$ for a.e. $t\in [0,1]$. In addition
    $$
    -|x-y|=-d_{\partial \Omega} (x)=f(1)-f(0)=\int_0^1 f'(t)\,dt,
    $$
which forces $f'(t)=-|x-y|$ for a.e. $t\in [0,1]$, since $-f'\leq |y-x|$. This gives
$$
d_{\partial \Omega} (x+t(y-x))=f(t)=f(0)-t|x-y|=d_{\partial \Omega} (x)-t|x-y|,
$$
and since we assumed $x$ to be a differentiability point for $d_{\partial \Omega}$, by letting $t\rightarrow 0^+$, we deduce
\begin{equation}
    \label{differential distance}
    (y-x)\cdot \nabla d_{\partial \Omega} (x)=-|y-x|.
\end{equation}
By Cauchy-Schwartz we get $1\leq |\nabla d_{\partial \Omega} (x)|$, which together with \eqref{dist_is_lip} proves $|\nabla d_{\partial \Omega}|=1$ a.e. in $\Omega$. 

Thus, let $x$ be any point of differentiability such that $|\nabla d_{\partial \Omega}(x)|=1$. Then \eqref{differential distance}, again by Cauchy-Schwartz, implies that $(y-x)$ and $\nabla d_{\partial \Omega}(x)$ are parallel, giving \eqref{gradient dist_expression}.
\end{proof}

We state and prove a result which will be used in the proof of \cref{P:normal trace is zero} to show that any divergence-free vector field $u\in BV(\Omega;\R^d)$, tangent to the boundary, has zero normal Lebesgue trace according to \cref{D:Leb normal trace}. To the best of the authors knowledge the theorem below is not standard in the literature. We thank Camillo De Lellis for the argument of the proof.

\begin{theorem}
    [Distance function and normal vector]\label{T:dist function and normal}
    Let $\Omega\subset \R^d$ be a bounded and Lipschitz domain. Then 
    \begin{equation}
        \label{lebsgue_point_for_nabla_d}
        \lim_{r\rightarrow 0^+}\frac{1}{r^d}\int_{\Omega \cap B_r(x)} |\nabla d_{\partial \Omega}(y)-n(x)|\,dy=0\qquad \text{for } \mathcal H^{d-1}\text{-a.e. } x\in \partial \Omega,
    \end{equation}
    where $n:\partial \Omega\rightarrow \mathbb S^{d-1}$ denotes the inward unit normal to $\partial \Omega$.
\end{theorem}
Let us notice that when $\Omega$ is piece-wise $C^2$ there is a shorter and quite classical argument to prove \eqref{lebsgue_point_for_nabla_d}. Indeed in this case, up to a $\mathcal H^{d-1}$-negligible set, $\partial \Omega$ is locally $C^2$, and thus $d_{\partial \Omega}\in C^2$ (locally) by the classical theory (see for instance \cite{GT}*{Lemma 14.16}). Then, a straightforward computation shows that every point $x\in \partial\Omega$ in the corresponding local neighbourhood is a Lebesgue point for $\nabla d_{\partial\Omega}$. As we will see, the proof for general Lipschitz domains is more delicate. Differently from \cref{L:dist funct}, here the Lipschitz regularity of $\partial \Omega$ plays a crucial role.
\begin{proof}

For sake of clarity, let us divide the proof into steps.    
 
1. \emph{Rewriting of \eqref{lebsgue_point_for_nabla_d}.} Since $\Omega$ is a Lipschitz domain, for $\mathcal{H}^{d-1}$-a.e. point $p=(p_1,\dots,p_{d-1},p_d)=:(\overline p , p_d)\in \partial\Omega$, there exists $r>0$, a Lipschitz map $\psi$ and a coordinate system such that $\partial \Omega\cap B_r(p)$ is the graph of $\psi$, with $\overline p$ a point of differentiability for $\psi$. After rotating the coordinates and possibly taking a smaller radius $r'<r$, we can assume $\nabla \psi (\overline p)=0$. In such a case we define the interior normal to $\partial\Omega$ in $p=(\overline p, p_d)$ as $n(p)=(0,\dots,0,1)=e_d$. Moreover, we can further assume that $\overline p=0$ and $p_d=\psi (\overline p)=0$, that is $p=0$. Thus, in the above notation, we have to prove that 
\begin{equation}
    \label{rewriting thesis}
          \lim_{r\rightarrow 0^+}\frac{1}{r^d}\int_{\Omega \cap B_r(0)} |\nabla d_{\partial \Omega}(x)-e_d|\,dx=0.
\end{equation}

2. \emph{Main claim.} We claim that $\forall \eta,\eps>0$ there exists $r_0=r_0(\eta,\eps)>0$ such that the following property holds. If $r<r_0$, $x\in B_r(0)\cap \left\{ |x_d|\geq \eta r\right\}$ and $y\in \partial \Omega$ is such that $d_{\partial \Omega}(x)=|x-y|$, then 
$$
\left| \frac{x-y}{|x-y|}-e_d\right|< \eps.
$$
Let us show how the claim gives \eqref{rewriting thesis}. We postpone the proof of the claim soon after. 

3. \emph{Proof of \eqref{rewriting thesis} assuming the claim.} Fix $\eta=\eps>0$. From \cref{L:dist funct} together with the above claim we have 
\begin{align*}
    \int_{\Omega\cap B_r(0)} |\nabla d_{\partial \Omega} (x)-e_d|\,dx&\leq 2 \mathcal H^d\left(B_r(0)\cap \left\{|x_d|\leq \eps r \right\}\right)\\
    &+  \int_{ B_r(0)\cap\left\{|x_d|\geq \eps r \right\}}\left| \frac{x-y(x)}{|x-y(x)|}-e_d\right|\,dx\\
    &\leq C \eps r^d,
\end{align*}
where $y(x)$ is any point on $\partial \Omega$ such that $d_{\partial \Omega} (x)=|x-y(x)|$. This proves that 
$$
 \limsup_{r\rightarrow 0^+}\frac{1}{r^d}\int_{\Omega \cap B_r(0)} |\nabla d_{\partial \Omega}(x)-e_d|\,dx\leq C\eps,
$$
from which \eqref{rewriting thesis} directly follows by the arbitrariness of $\eps>0$.

4. \emph{Proof of the claim.} We run the proof by contradiction. If the conclusion would be false, we could find $\eps,\eta>0$ and $x_k,y_k, r_k$ such that
\begin{itemize}
    \item[$(i)$] $x_k=(\overline x_k, (x_k)_d)\in \Omega \cap B_{r_k}(0)$, with $r_k\rightarrow 0^+$ as $k\rightarrow\infty$;
    \item[$(ii)$] $|( x_k)_d|\geq \eta r_k$;
    \item[$(iii)$] $y_k=(\overline y_k,(y_k)_d)\in \partial \Omega$ such that $d_{\partial \Omega}(x_k)=|y_k-x_k|$; 
    \item[$(iv)$] $\left| \frac{x_k-y_k}{|x_k-y_k|}-e_d\right|\geq \eps$.
\end{itemize}
Since $0\in \partial \Omega$, then $|y_k-x_k|\leq |x_k|\leq r_k$. Moreover, $(y_k)_d=\psi ( \overline y_k)$, and from $\nabla \psi (0)=0$, $\psi (0)=0$ and $|\overline y_k|\leq |y_k|\leq |y_k-x_k|+|x_k|\leq 2r_k$, we have 
\begin{equation}
    \label{lim_1_is_0}
    \lim_{k\rightarrow \infty}\frac{(y_k)_d}{r_k}= \lim_{k \rightarrow \infty} \frac{\psi(\overline{y}_k) - \psi(0)}{r_k} = 0.
\end{equation}
Thus, by $(ii)$ we obtain
\begin{align} \label{lower_bound_x_k-y_k}
    |x_k-y_k|&\geq |(x_k)_d-(y_k)_d|\geq \eta r_k \left(1- \frac{|(y_k)_d|}{\eta r_k}  \right)\geq \frac{\eta}{2}r_k,
\end{align}
if $k$ is chosen large enough. Moreover, by $(ii)$ and \eqref{lim_1_is_0}, we also have
\begin{equation}\label{y over x vanish}
\left| \frac{(y_k)_d}{(x_k)_d}\right|\leq \frac{|(y_k)_d|}{\eta r_k}\rightarrow 0 \qquad \text{as } k\rightarrow \infty.
\end{equation}
In particular, \eqref{lim_1_is_0} and \eqref{lower_bound_x_k-y_k} imply that
$$
\lim_{k\rightarrow \infty}\frac{(y_k)_d}{|x_k-y_k|}=0.
$$
Then, together with \eqref{y over x vanish}, we infer the expansion
$$
\frac{x_k-y_k}{|x_k-y_k|}=\frac{(\overline x_k-\overline y_k,(x_k)_d)}{\sqrt{|\overline x_k -\overline y_k|^2+ |(x_k)_d|^2}} + o(1).
$$
Thus
\begin{align*}
    \left|\frac{x_k-y_k}{|x_k-y_k|} - e_d\right|&\leq \frac{|\overline x_k-\overline y_k|}{\sqrt{|\overline x_k-\overline y_k|^2+|(x_k)_d|^2}}+ \left|\frac{(x_k)_d}{\sqrt{|\overline x_k-\overline y_k|^2+|(x_k)_d|^2}} -1\right|+o(1)\\
    &\leq \frac{|\overline x_k-\overline y_k|}{|(x_k)_d|}\frac{1}{\sqrt{1+\frac{|\overline x_k-\overline y_k|^2}{|(x_k)_d|^2}}}+ \left|\frac{1}{\sqrt{1+\frac{|\overline x_k-\overline y_k|^2}{|(x_k)_d|^2}}}-1 \right| + o(1).
\end{align*}
Then $(iv)$ implies that necessarily 
\begin{equation}\label{almost final contrad}
\liminf_{k\rightarrow \infty} \frac{|\overline x_k- \overline y_k|}{ (x_k)_d}\geq \tilde \eps,
\end{equation}
for some $\tilde \eps =\tilde \eps (\eta,\eps)>0$. But \eqref{almost final contrad} contradicts the minimality of $y_k$. Indeed, letting $z_k:=(\overline x_k,\psi (\overline x_k))$, since $\nabla \psi (0)=0$ and $\psi (0)=0$, we have 
\begin{equation}
    \label{lim_equal_1}
     \lim_{k\rightarrow \infty} \frac{|x_k-z_k|}{(x_k)_d}=\lim_{k\rightarrow \infty} \frac{|(x_k)_d-\psi (\overline x_k)|}{(x_k)_d}=1.
\end{equation}
On the other hand, \eqref{almost final contrad} implies 
\begin{equation}
    \label{final contrad}
    \lim_{k\rightarrow \infty} \frac{|x_k-y_k|}{(x_k)_d}=\liminf_{k\rightarrow \infty} \frac{\sqrt{|\overline x_k-\overline y_k|^2 + |(x_k)_d|^2}}{(x_k)_d}\geq \sqrt{1+\tilde \eps ^2}>1,
\end{equation}
where in the first equality we have also used \eqref{y over x vanish}.
It is clear that \eqref{lim_equal_1}, together with \eqref{final contrad}, yields to a contradiction, since 
$$
|x_k-y_k|=d_{\partial \Omega}(x_k)\leq |x_k-z_k|.
$$
\end{proof}

\subsection{Minkowski content and rectifiable sets}
For all $m\geq 1$, we let $\omega_m:=\mathcal{H}^m (B_1)$ to be the $m$-dimensional volume of the unit ball $B_1\subset \R^m$. We also keep the notation $(A)_\eps$ introduced in \eqref{eps_tub_neig_def} for the $\eps$-tubular neighbourhood.
\begin{proposition}
    [Minkowski content of rectifiable sets]\label{P:federer minkowski rectifiable}
    Let $\Omega\subset \R^d$ be a bounded domain with Lipschitz boundary $\partial\Omega$. If $C\subset \partial \Omega$ is closed in the induced topology of $\partial \Omega$, then 
    \begin{equation}\label{minkow=hausd}
        \lim_{\eps\rightarrow 0^+} \frac{\mathcal H^d((C)_\eps)}{\omega_1 \eps}=\mathcal H^{d-1} (C).
    \end{equation}
    Moreover, there exists a dimensional constant $c>0$ such that 
    \begin{equation}
        \label{lower bound half density point}
        \liminf_{\eps\rightarrow 0^+}  \frac{\mathcal H^d((C)_\eps\cap \Omega)}{\eps}\geq c \mathcal H^{d-1} (C).
    \end{equation}
\end{proposition}
\begin{proof}
Since $\partial \Omega$ is $(d-1)$-rectifiable and $C$ is closed in $\R^d$, \eqref{minkow=hausd} follows by \cite{Fed}*{Theorem 3.2.39}. 

Let $\eps_k\rightarrow 0^+$ be any sequence. Since $\partial \Omega$ is Lipschitz, by \cite{AFP00}*{Theorem 3.61} there exists $C_{\sfrac{1}{2}}\subset C$, with $\mathcal H^{d-1} (C\setminus C_{\sfrac{1}{2}})=0$, such that 
$$
\lim_{k\rightarrow \infty}f_{\eps_k} (x):= \lim_{\eps\rightarrow 0^+}\frac{\mathcal H^d(B_{\eps_k}(x)\cap \Omega)}{\mathcal H^d(B_{\eps_k} (x))}=\frac{1}{2}\qquad \forall x\in C_{\sfrac{1}{2}}.
$$
Pick any $\delta>0$, which will be fixed at the end of the proof. By Egorov's theorem we can find $A_\delta\subset C_{\sfrac{1}{2}}$, closed in the induced topology of $\partial \Omega$, such that $\mathcal H^{d-1} ( C_{\sfrac{1}{2}}\setminus A_\delta)<\delta$ and 
\begin{equation}
   \label{unif_conv_dens_half}
   \lim_{k\rightarrow \infty} \sup_{x\in A_\delta} \left| f_{\eps_k} (x)-\frac{1}{2}\right|=0.
\end{equation}
Consider the covering 
$$
A_\delta\subset \bigcup_{x\in A_\delta} B_\eps (x)=(A_\delta)_\eps,
$$
and extract a disjoint Vitali subcovering $\left\{ B_\eps (x_i)\right\}_{i=1}^{N_\eps}$ such that 
$$
(A_\delta)_\eps\subset \bigcup_{i=1}^{N_\eps}B_{5\eps}(x_i).
$$
Thus, since $A_\delta\subset \R^d$ is closed and $(d-1)$-rectifiable, by \eqref{minkow=hausd} we obtain 
\begin{equation}\label{number convering lower bound}
\mathcal{H}^{d-1} (A_\delta)=\lim_{\eps\rightarrow 0^+} \frac{\mathcal H^d((A_\delta)_\eps)}{\omega_1 \eps}\leq \liminf_{\eps\rightarrow 0^+} \frac{\sum_{i=1}^{N_\eps}\mathcal H^d(B_{5\eps}(x_i))}{\omega_1 \eps}\leq c \liminf_{\eps\rightarrow 0^+} N_\eps \eps^{d-1},
\end{equation}
for some dimensional constant $c>0$. By \eqref{unif_conv_dens_half} and since the family $\left\{ B_\eps (x_i)\right\}_{i=1}^{N_\eps}$ is disjoint, we find $k_0\in \N$  sufficiently large (recall that $\eps_k\rightarrow0^+$) such that, for all $k\geq k_0$, we bound from below
\begin{align*}
    \frac{\mathcal H^d((C)_{\eps_k}\cap \Omega)}{\eps_k}&\geq  \frac{\mathcal H^d((A_\delta)_{\eps_k}\cap \Omega)}{\eps_k}\geq \frac{\mathcal H^d\left(\bigcup_{i=1}^{N_{\eps_k}} B_{\eps_k} (x_i)\cap \Omega\right)}{\eps_k}\\
    &=\sum_{i=1}^{N_{\eps_k}} \frac{\mathcal H^d (B_{\eps_k} (x_i)\cap \Omega)}{\eps_k}=\omega_d \eps_k^{d-1}\sum_{i=1}^{N_{\eps_k}} \frac{\mathcal H^d (B_{\eps_k} (x_i)\cap \Omega)}{\mathcal H^d (B_{\eps_k}(x_i))}\\
    &\geq \frac{\omega_d}{4} N_{\eps_k}\eps_k^{d-1}.
\end{align*}
In particular, by \eqref{number convering lower bound} we obtain 
$$
\liminf_{k\rightarrow \infty} \frac{\mathcal H^d((C)_{\eps_k}\cap \Omega)}{\eps_k}\geq c \mathcal H^{d-1} (A_\delta)\geq c \left( \mathcal H^{d-1} (C)-\delta\right).
$$
Hence, since the sequence $\eps_k\rightarrow 0^+$ was arbitrary, we obtain \eqref{lower bound half density point} by choosing $2\delta=\mathcal H^{d-1}(C)$.
\end{proof}

\subsection{A one-sided Egorov lemma}
We have the following “one-sided" Egorov-type convergence result.
\begin{lemma}[One-sided Egorov]\label{L:egorov_one_side}
    Let $(X,\mu)$ be a finite measure space, $L\in \R$ and $f_n$ a sequence of measurable functions such that $\liminf_{n\rightarrow\infty} f_n (x) \geq L$ for $\mu$-a.e. $x\in X$. For every $\delta>0$ there exists $A_\delta\subset X$ such that $\mu (A_\delta^c)<\delta$ and 
    $$
    \liminf_{n\rightarrow \infty} \inf_{x\in A_\delta} f_n(x)\geq L.
    $$
\end{lemma}

\begin{proof}
Since $\liminf_{n\rightarrow\infty} f_n \geq L$ $\mu$-a.e. in $X$, for every $k\in \N$, the set 
$$
\bigcup_{N=1}^\infty \bigcap_{n\geq N} \left\{x\,:\, f_n(x)\geq L-\frac{1}{k}\right\}
$$
has full measure in $X$. For every  $k$ fixed, the sets
$$
A_{N,k}:=\bigcap_{n\geq N} \left\{x\,:\, f_n(x)\geq L-\frac{1}{k}\right\}
$$
form an increasing family with respect to $N\geq 1$, and since $\mu(X)<\infty$, we deduce that
$$
0 = \mu \left( \bigcap_{N = 1}^{\infty} A_{N,k}^c\right) = \lim_{N\rightarrow \infty} \mu(A_{N,k}^c).
$$
Fix $\delta>0$. For every $k\in \N$ we can find $N_k=N_k(\delta)\in \N$ such that $\mu(A_{N_k,k}^c)<\delta 2^{-k}$,
from which
$$
\mu \left( \bigcup_{k=1}^\infty A_{N_k,k}^c\right)<\delta.
$$
Thus define $A_\delta:= \bigcap_{k=1}^\infty A_{N_k,k}$. Clearly $A_\delta \subset A_{N_m,m}$ for every $m\in \N$. Thus let $x\in A_\delta$ and $m\in \N$ be arbitrary. In particular $x\in A_{N_m,m}$, that is 
$$
f_n(x)\geq L-\frac{1}{m}\qquad \forall n\geq N_m.
$$
Since $N_m$ does not depend on $x$, we deduce
$$
\inf_{n\geq N_m}\inf_{x\in A_\delta}f_n(x)\geq L-\frac{1}{m},
$$
which implies
$$
\liminf_{n\rightarrow \infty}\inf_{x\in A_\delta}f_n(x)\geq L-\frac{1}{m}.
$$
The proof is concluded since $m\in \N$ was arbitrary.
\end{proof}

\subsection{Alberti--Ambrosio anisotropic kernel optimization}
Here we prove a variation of the so called Alberti's lemma, which was introduced to prove the renormalization property of weak solutions to the transport equations with $BV$ vector fields, first established by Ambrosio in \cite{Ambr04}. Actually, in the original proof, Ambrosio performed the optimization of the kernel by exploiting the rank-one structure of the singular part of the gradient of $BV$ function, a deep result in Geometric Measure Theory also due to Alberti \cite{Al93}. Alberti's lemma provides a shortcut to optimize the choice of the kernel without relying on too deep theories. 

Note that the kernels used in the next lemma are compactly supported, but not necessarily in $B_1$, and not smooth.
\begin{lemma} [Alberti] \label{L:alberti}
Let $\eta: \R^d \to \R^d$ be a Lipschitz  vector field. Assume that $\eta$ is odd and $\div \eta$ is a constant. Let 
$$\mathcal{K}_c = \left\{ \rho \in W^{1,\infty}_c(\R^d), \,\rho \geq 0, \,\rho \text{ even}, \int_{\R^d} \rho(z) \, dz =1  \right\}. $$
Then, we have that 
\begin{equation}
\inf_{\rho \in \mathcal{K}_c} \int_{\R^d} \abs{\nabla \rho(z) \cdot \eta (z)} \, dz = |\div \eta|.  
\end{equation}  
\end{lemma}
Usually Alberti's lemma (see for instance \cite{Cr09}*{Lemma 2.13}) assumes that $\eta(z)=Mz$, for a given matrix $M$, and proves that the infimum equals $|\tr M|$. Here we preferred to state the version with a general odd Lipschitz vector field $\eta$ since it might be of better use in contexts in which the linearity in $z$ is missing. For instance, this might be the case in proving energy conservation for incompressible Euler in an Onsager's critical class different from $BV\cap L^\infty$.
\begin{proof} 
Letting $X : [0,+\infty) \times \R^d \to \R^d$ be the flow generated by the vector field $\eta$, by the classical Cauchy--Lipschitz theory we have that $X_t$ is a bi-Lipschitz transformation with Lipschitz constants proportional to $e^{Lt}$, where $L$ is the Lipschitz constant of $\eta$. Since $\eta$ is odd we infer that $X_t, X_t^{-1}$ are odd for any time slice $t \in [0, +\infty)$. We split the proof in the two cases $\div \eta=0$ and $\div \eta\neq 0$.

\emph{$1.$ Case $\div \eta=0$}. Since $\eta$ is incompressible, then $X_t, X_t^{-1}$ are measure preserving. Fix a convolution kernel $\theta \in \mathcal{K}$ (see \eqref{kernel_in_B1}) and, for any $T>0$, define 
\begin{equation}
\rho_T(z) := \frac{1}{T} \int_0^T \theta(X_t^{-1}(z))\, dt. 
\end{equation}
It is immediate to see that $\rho_T$ is even, non-negative and Lipschitz, with Lipschitz constant proportional to $e^{LT}$ and $\rho_T$ has compact support. Moreover, we have that 
$$\int_{\R^d} \rho_T(z) \, dz = \frac{1}{T} \int_0^T \int_{\R^d} \theta(X_t^{-1}(z))\, dz \, dt = \frac{1}{T} \int_0^T \int_{\R^d} \theta(z)\, dz \, dt = 1. $$
Therefore, we have that $\rho_T \in \mathcal{K}_c$. Since $X$ is the flow generated by $\eta$, we have that 
\begin{align*}
\nabla \rho_T(z) \cdot\eta(z) & = \frac{1}{T} \int_0^T \nabla (\theta \circ X_t^{-1}) (z) \cdot \eta(z) \, dt
\\ & = \frac{1}{T} \int_0^T \nabla \theta(X_t^{-1}(z)) \nabla X_t^{-1}(z) \cdot \eta(z) \, dt 
\\ & = \frac{1}{T} \int_0^T \nabla \theta(X_t^{-1}(z)) \nabla X_t^{-1}(z) \cdot\partial_t X_t(X_t^{-1}(z)) \, dt.
\end{align*}
Moreover
$$0= \frac{d}{dt} (X_t \circ X_t^{-1})(z) = \partial_t X_t(X_t^{-1}(z)) + \nabla X_t(X_t^{-1}(z)) \partial_t X_t^{-1}(z).  $$
Hence, we infer that 
\begin{align}
\nabla \rho_T(z)\cdot  \eta(z) & = - \frac{1}{T} \int_0^T \nabla \theta(X_t^{-1}(z)) \nabla X_t^{-1}(z) \cdot \nabla X_t (X_t^{-1}(z)) \partial_t X_t^{-1}(z) \, dt\nonumber 
\\ & = - \frac{1}{T} \int_0^T \nabla \theta (X_t^{-1}(z))\cdot \partial_t X_t^{-1}(z) \, dt\nonumber 
\\ & = -\frac{1}{T} \int_0^T \frac{d}{dt} (\theta \circ X_t^{-1}(z)) \, dt \nonumber 
\\ & = \frac{1}{T} \left[ \theta(z) -\theta(X_T^{-1}(z)) \right]. \label{comput_alberti}
\end{align}
Therefore, we conclude that 
\begin{equation}
\int_{\R^d} \abs{\nabla \rho_T(z) \cdot\eta(z) }\, dz \leq \frac{1}{T} \int_{\R^d} \left[ \theta(z) + \theta(X_T^{-1}(z)) \right] \,dz = \frac{2}{T}.  
\end{equation}
Letting $T \to +\infty$, we conclude the proof.

\emph{$2.$ Case $\div \eta\neq 0$}. Denote $\div \eta=a\in \R$. The Jacobian $J_t(z):=\det (\nabla X_t (z))$ solves
$$
\left\{\begin{array}{l}
\frac{d}{dt} J_t(z) = a J_t (z)\\ 
J_0(z)= 1,
\end{array}\right.
$$
which gives $J_t(z) = e^{at}$ for any $z$. Thus, for a fixed convolution kernel $\theta\in \mathcal K$, for any $T>0$, we define 
$$
\rho_T (z):= \frac{a}{e^{aT}-1} \int_0^T \theta(X_t^{-1}(z))\,dt.
$$
We have $\rho_T\geq 0$ and 
$$
\int_{\R^d} \rho_T(z)\,dz=\frac{a}{e^{aT}-1} \int_0^T J_t\,dt \int_{\R^d}\theta(z)\,dz=1.
$$
Therefore $\rho_T\in \mathcal K_c$. The same computations \eqref{comput_alberti} lead to
$$
\int_{\R^d} \abs{\nabla \rho_T(z) \cdot\eta(z) }\, dz \leq\frac{|a|}{\left| e^{aT}-1\right|}\left( 1+ e^{aT}\right).
$$
In both cases $a>0$ and $a<0$ the right hand side in the above estimate converges to $|a|$ as $T\rightarrow +\infty$, which together with the trivial lower bound
$$
\int_{\R^d} \abs{\nabla \rho_T(z) \cdot\eta(z) }\, dz \geq \left| \int_{\R^d} \nabla \rho_T(z) \cdot\eta(z) \,dz\right|=|a|,
$$
concludes the proof.
\end{proof}

Following the idea used in \cite{Ambr04}, a consequence of \cref{L:alberti} is the following result, which will be used  in \cref{S:BV_cons} to prove the energy conservation in $L^1_tBV_x\cap L^\infty_{x,t}$ of \cref{T:cons_local_BV}. The set $\mathcal{K}$ denotes kernels supported in the unit ball as introduced in \eqref{kernel_in_B1}.

\begin{proposition}\label{P:Local energy conservation}
Let $U\subset \R^m$ be an open set. Take $\lambda \in  \mathcal{M}_{\rm loc}(U)$ a signed measure, $\nu\in \mathcal{M}_{\rm loc}(U)$ a positive measure and $\eta:U\times B_1\rightarrow \R^d$ a vector field with the following properties.
\begin{itemize}
    \item[(i)] Lipschitz continuity: $\eta_{x}\in \Lip(B_1)$ for $\nu$-a.e. $x\in U$;
    \item[(ii)] Odd and $1$-homogeneous: $\eta_{x} (a z)=a\eta_{x}( z)$ for all $a\in (-|z|^{-1},|z|^{-1})$ and all $z\in \overline B_1$, for $\nu$-a.e. $x\in U$;
\item[(iii)] Divergence-free: $\div_z \eta_{x}=0$ for $\nu$-a.e. $x\in U$;
    \item[(iv)] Uniformly $\nu$-integrable: $\| \eta_{x}(\cdot)\|_{L^\infty(B_1)}\in L^1_{\rm loc} (U;\nu)$.
\end{itemize}
If for any $O\subset\joinrel\subset K\subset\joinrel\subset  U$, $O$ open and $K$ compact, there exists a constant $C=C(K)>0$ such that
\begin{equation}\label{cons_cond_general}
 \abs{ \langle\lambda,\varphi \rangle  }  \leq C \|\varphi\|_{L^\infty(O)}  \int_O \left(  \int_{B_1} \abs{\nabla \rho(z) \cdot\eta_{x}(z)} \, dz  \right) d \nu \qquad \forall (\varphi,\rho) \in C^0_c(O)\times \mathcal{K},
\end{equation}
then 
\begin{equation}\label{inf=0}
|\lambda| (K)=0 \qquad \forall K\subset\joinrel\subset U\text{ compact}.
\end{equation}
\end{proposition}

\begin{proof}

The proof is an application of \cref{L:alberti}. Fix any $K\subset \joinrel \subset U$ and $\rho\in \mathcal K$. First, having \eqref{cons_cond_general} for all $\varphi \in  C^0_c(O) $, for any open set $O\subset\joinrel\subset K$, it implies the inequality (in the compact set $K$, which fixes the constant $C$) in the sense of measures
\begin{equation*}
|\lambda |\leq C\left( \int_{B_1} \abs{\nabla \rho(z) \cdot\eta_{x}(z)} \, dz\right) \nu.
\end{equation*}
In particular $|\lambda|\ll \nu$ and, denoting by $D_\nu |\lambda|$ the Radon--Nikodym derivative of $|\lambda|$ with respect to $\nu$, we have 
\begin{equation}\label{rad_nik_der}
(D_\nu |\lambda|)(x)\leq C\int_{B_1} \abs{\nabla \rho(z) \cdot\eta_{x}(z)} \, dz\qquad \text{for } \nu\text{-a.e. } x\in K.
\end{equation}
Set 
$$
\mathcal{K}_{W} := \left\{ \rho \in W^{1,1}_0(B_1), \,\rho \geq 0, \,\rho \text{ even }, \int_{\R^d} \rho(z) \, dz =1  \right\}.
$$
We claim that 
\begin{equation}\label{variation inequal 2}
(D_\nu |\lambda |)(x)\leq C \inf_{\rho\in \mathcal K_W} \int_{B_1} \abs{\nabla \rho(z) \cdot\eta_{x}(z)} \, dz\qquad \text{for } \nu\text{-a.e. } x\in K.
\end{equation}
Indeed, for any $\rho \in \mathcal K_W$ there exists a sequence $\{\rho_n\}_n \subset \mathcal K$ such that $\rho_n\rightarrow \rho$ in $W^{1,1}_0(B_1)$. Thus, \eqref{rad_nik_der} remains valid for every $\rho \in \mathcal K_W$. Then, pick $\tilde{\mathcal K}_W\subset \mathcal K_W$ a countable dense subset (with respect to the $W^{1,1}$ topology). For any $\rho \in \Tilde{\mathcal{K}}_W$, we find a $\nu$-negligible set $A_\rho \subset K$ such that \eqref{rad_nik_der} is satisfied for any $x \in A_\rho^c$. Since $\Tilde{\mathcal K}_W$ is countable, we have that $A = \bigcup_{\rho \in \Tilde{\mathcal K}_W} A_\rho$ is $\nu$-negligible and for any $x \in A^c$ we have that 
$$(D_\nu |\lambda |)(x)\leq C \inf_{\rho\in \Tilde{ \mathcal K}_W} \int_{B_1} \abs{\nabla \rho(z) \cdot\eta_{x}(z)} \, dz. $$
Moreover, since the functional $\rho \mapsto \int_{B_1} \abs{\nabla \rho(z)\cdot\eta_x(z)}\, dz$ is continuous with respect to the $W^{1,1}$ topology, by density we have that 
$$
 \inf_{\rho\in \tilde{\mathcal K}_W} \int_{B_1} \abs{\nabla \rho(z) \cdot\eta_{x}(z)}\,dz = \inf_{\rho\in \mathcal K_W} \int_{B_1} \abs{\nabla \rho(z) \cdot\eta_{x}(z)}\,dz \qquad \forall x \in K. 
$$
Thus, \eqref{variation inequal 2} is proved. 

From now on, fix $x \in K$ such that \eqref{variation inequal 2} is satisfied. For any $\eps>0$, \cref{L:alberti} provides a $\rho_\eps \in \mathcal K_c$ such that 
$$
\int_{\R^d} |\nabla \rho_\eps (z) \cdot \eta_{x} (z)|\,dz<\eps.
$$
Since $\spt \rho_\eps$ is not necessarily contained in $B_1$, that is $\rho\not \in \mathcal K_W$, we have to consider a rescaling, which, together with the $1$-homogeneity of $\eta_{x}$, will conclude the proof. Indeed, if $B_{R_\eps}$ is a ball containing the support of $\rho_\eps$, we define
$$
\tilde \rho_\eps (z):=R^d_\eps \rho_\eps(R_\eps z).
$$
Clearly $ \tilde \rho_\eps \in \mathcal K_W$, and moreover by $(ii)$
$$
\int_{B_1} |\nabla \tilde\rho_\eps (z) \cdot \eta_{x} (z)|\,dz=R_\eps\int_{\R^d} \left|\nabla \rho_\eps (w) \cdot \eta_{x} (wR_\eps^{-1})\right|\,dw=\int_{\R^d} |\nabla \rho_\eps (w) \cdot \eta_{x} (w)|\,dw<\eps.
$$
Together with \eqref{variation inequal 2}, and being $\eps>0$ arbitrary, this proves that 
$$
(D_\nu |\lambda|) (x)=0\qquad \qquad \text{for } \nu\text{-a.e. } x\in K,
$$
and the arbitrariness of $K\subset \joinrel\subset U$ concludes the proof.
\end{proof}

\section{Duchon--Robert measure in critical Onsager's classes}\label{S:DR_formula}
This section is devoted to the proof of \cref{T:ons_crit} and discuss some of its consequences.
\begin{proof}[Proof of \cref{T:ons_crit}]

By looking at the expression of $D_\eps[u]$ in \eqref{DR_eps_formula}, it is natural to define the trilinear operator $T(u,v,w):= u (v\cdot w)$. Then, we define
$$
T_{\eps,z} (x,t):=\frac{1}{\eps} T\left( \delta_{\eps z} u , \delta_{\eps z} u , \delta_{\eps z} u\right).
$$
Pick any $K:=A\times I\subset \joinrel\subset \Omega\times (0,T)$. By H\"older inequality it is clear that under any of the assumptions \eqref{Besov_assumpt}-\eqref{BV_assumpt} we have
\begin{equation}
    \label{bounded L1 loc}
    \sup_{\substack{  |z|\leq 1 \\ \eps>0}} 
    \|T_{\eps,z}\|_{L^1(K)}\lesssim 1.
\end{equation}
Moreover, by the trilinear structure of $T_{\eps,z}$ we can deduce equicontinuity of the sequence with respect to $z$. Indeed we have
\begin{align}
    \eps\left( T_{\eps,z_1}-T_{\eps,z_2}\right)&= T\left( \delta_{\eps z_1} u-\delta_{\eps z_2} u, \delta_{\eps z_1} u,\delta_{\eps z_1} u\right)+ T\left( \delta_{\eps z_2} u, \delta_{\eps z_1} u-\delta_{\eps z_2} u, \delta_{\eps z_1} u\right)\nonumber\\
    &+ T\left( \delta_{\eps z_2} u, \delta_{\eps z_2} u, \delta_{\eps z_1} u -\delta_{\eps z_2} u\right).\label{trilin expansion}
\end{align}
The three terms in the above formula have to be estimated in a different, although quite similar, way depending on which is the assumption on $u$ among \eqref{Besov_assumpt}-\eqref{BV_assumpt}. Thus let us divide the cases.

\emph{Assuming \eqref{Besov_assumpt} for $p\leq 3$}:
Note that $\delta_{\eps z_1} u - \delta_{\eps z_2} u= u(x+\eps z_1,t)-u(x+\eps z_2,t)$. By the definition of Besov seminorm \eqref{besov_norm}, we estimate each term in the right hand side of \eqref{trilin expansion} by H\"older inequality, always taking $\delta_{\eps z_1} u - \delta_{\eps z_2} u$ in $L^p$ and each $\delta_{\eps z_i} u$ in $L^{2p'}$, getting
\begin{align}
\left\| T_{\eps,z_1}-T_{\eps,z_2}\right\|_{L^1(K)} &\lesssim |z_1 - z_2|^{\frac{1}{p}}\left(|z_1|^{\frac{1}{p'}}+ |z_1|^\frac{1}{2p'}|z_2|^\frac{1}{2p'}+ |z_2|^{\frac{1}{p'}}\right)\nonumber\\
&\lesssim |z_1 - z_2|^{\frac{1}{p}}\left(|z_1|^{\frac{1}{p'}}+ |z_2|^{\frac{1}{p'}}\right),\label{est hp Besov p}
\end{align}
where in the last line we used Young's inequality $2ab\leq a^2+b^2$.

\emph{Assuming \eqref{Besov_assumpt} for $p> 3$}:
In this case we always take $\delta_{\eps z_1} u - \delta_{\eps z_2} u$ in $L^{2p'}$, one of the $\delta_{\eps z_i} u$ in $L^{2p'}$ and the other in $L^p$. Thus, by the definition of Besov seminorm \eqref{besov_norm}, we have
\begin{align}
\left\| T_{\eps,z_1}-T_{\eps,z_2}\right\|_{L^1(K)} &\lesssim |z_1 - z_2|^{\frac{1}{2p'}}\left(|z_1|^{\frac{p+1}{2p}}+ |z_1|^\frac{1}{p}|z_2|^\frac{1}{2p'}+ |z_2|^{\frac{p+1}{2p}}\right)\nonumber\\
&\lesssim |z_1 - z_2|^{\frac{1}{2p'}}\left(|z_1|^{\frac{p+1}{2p}}+ |z_2|^{\frac{p+1}{2p}}\right).\label{est hp Besov p bigger 3}
\end{align}
The last inequality follows by Young's inequality $ab\leq \frac{a^r}{r}+\frac{b^s}{s}$, with $r=\frac{p+1}{2}$ and $s=\frac{p+1}{p-1}$.

\emph{Assuming \eqref{1_2_Assumpt}}:
In this case we always put $\delta_{\eps z_1} u - \delta_{\eps z_2} u$ in $L^{2}$, one of the $\delta_{\eps z_i} u$ in $L^{2}$ and the other in $L^\infty$. Hence, using again the definition of Besov seminorm \eqref{besov_norm}, we have 
\begin{align}
\left\| T_{\eps,z_1}-T_{\eps,z_2}\right\|_{L^1(K)} \lesssim |z_1 - z_2|^{\frac{1}{2}}\left(|z_1|^{\frac{1}{2}}+ |z_2|^{\frac{1}{2}}\right).\label{est hp Besov 2}
\end{align}

\emph{Assuming \eqref{BV_assumpt}}:
In this case we estimate $\delta_{\eps z_1} u - \delta_{\eps z_2} u$ in $L^{1}$ by \cref{L:BD_gradient}, and both of the $\delta_{\eps z_i} u$ in $L^{\infty}$, getting
\begin{align}
\left\| T_{\eps,z_1}-T_{\eps,z_2}\right\|_{L^1(K)} \lesssim |z_1 - z_2|.\label{est hp BV}
\end{align}

To summarize, under any of the assumption of the theorem, we have that  
\begin{equation}
    \label{est equicontinuity}
    \left\| T_{\eps,z_1}-T_{\eps,z_2}\right\|_{L^1(K)} \lesssim |z_1 - z_2|^\gamma,
\end{equation}
for some $\gamma \in (0,1]$ depending on the regularity of $u$, where the implicit constant depends on $K, u$ and it is uniform with respect to $\eps$.

We are ready to apply an Ascoli--Arzelà argument. Pick $\{z_k\}_k\subset B_1$ dense. By \eqref{bounded L1 loc} we have $T_{\eps,z_1}$ bounded in $L^1_{\rm loc}(\Omega\times (0,T);\R^d)$, thus we can find a subsequence $T_{j,z_1}\rightharpoonup \mu_{z_1}$ for some $\mu_{z_1}\in \mathcal M_{\rm loc} (\Omega\times (0,T);\R^d)$. Evaluating the above subsequence on $z_2$, i.e. $T_{j,z_2}$, we can extract a further (non relabelled)  subsequence such that $T_{j,z_2}\rightharpoonup \mu_{z_2}$ for some $\mu_{z_2}\in \mathcal M_{\rm loc} (\Omega\times (0,T);\R^d)$. Hence, by a diagonal argument, we find a subsequence $\{T_{n,z}\}_{n}$ and measures $\mu_{z_k} \in \mathcal{M}_{\rm loc}(\Omega \times (0,T); \R^d)$, such that 
$$
T_{n,z_k}\rightharpoonup \mu_{z_k}\qquad \forall k.
$$
We aim to extend $\{\mu_{z_k}\}_k$ to a map $\mu \in C^0(\overline{B}_1; \mathcal{M}_{\rm loc}(\Omega \times (0,T); \R^d))$ such that $T_{n,z} \rightharpoonup \mu_z$ for any $z \in B_1$.   Given any test function $\varphi \in C^0_c(\Omega \times (0,T))$, by a standard argument we check that $\langle T_{n,z},\varphi\rangle$ is a Cauchy sequence in $C^0(\overline B_1;\R^d)$. Indeed, fix $\e>0$ and by the equicontinuity estimate \eqref{est equicontinuity} find $\delta >0$ such that 
$$\abs{\langle T_{n,z}, \varphi \rangle - \langle T_{n,w}, \varphi \rangle } < \e \qquad \forall n\in \N \qquad \text{if } \abs{z-w} < \delta. $$
Clearly the family $\{B_\delta (z_k)\}_k$ covers $\overline B_1$. By compactness we find $N \in \N$ such that $\overline B_1 \subset \bigcup_{i=1}^N B_{\delta}(z_i)$. Then, by the weak convergence on $\{z_k\}_k$, we find $M \in \N$ such that 
$$\abs{ \langle T_{n_1, z_i}, \varphi \rangle - \langle T_{n_2, z_i}, \varphi \rangle} < \e \qquad \forall n_1, n_2 \geq M, \qquad \forall i = 1, \dots, N. $$
Hence, for any $z \in \overline{B}_1$, pick $j \in \{1, \dots, N\}$ such that $\abs{z-z_j}< \delta$. Thus, for $n_1, n_2 \geq M$ we have that 
\begin{align*}
    \abs{\langle T_{n_1, z} , \varphi \rangle - \langle T_{n_2, z}, \varphi \rangle } & \leq \abs{\langle T_{n_1, z} , \varphi \rangle - \langle T_{n_2, z_j}, \varphi \rangle } + \abs{\langle T_{n_1, z_j} , \varphi \rangle - \langle T_{n_2, z_j}, \varphi \rangle } 
    \\  & + \abs{\langle T_{n_2, z_j} , \varphi \rangle - \langle T_{n_2, z}, \varphi \rangle } < 3 \e. 
\end{align*}

Thus, there exists a vector $V_\varphi \in C^0(\overline B_1;\R^d)$ such that 
\begin{equation}
    \label{uniform conv to a vector}
    \lim_{n\rightarrow \infty} \sup_{|z|\leq 1} \left|\langle T_{n,z},\varphi\rangle-V_\varphi (z) \right|=0.
\end{equation}
The map $\varphi \mapsto V_\varphi (z)$ is linear and satisfies 
$$
\left| V_\varphi (z)\right| \lesssim \|\varphi \|_{L^\infty}\qquad \forall \varphi \in C^0_c(\Omega\times (0,T)).
$$
By Riesz theorem we find $\mu_z\in \mathcal M_{\rm loc} (\Omega\times (0,T);\R^d)$ such that 
$$
 V_\varphi (z)=\langle \mu_z,\varphi\rangle.
$$
It is immediate to see that $\mu_z$ is an extension of $\{\mu_{z_k}\}_k$.  Then, \eqref{uniform conv to a vector} proves that $\langle T_{n,z},\varphi\rangle\rightarrow \langle \mu_z,\varphi\rangle$ in $C^0(\overline B_1;\R^d)$, from which
$$
\langle D[u],\varphi\rangle= \frac14 \lim_{n\rightarrow \infty} \int_{B_1} \nabla \rho (z) \cdot T_{n,z} \,dz=\frac14 \int_{B_1} \nabla \rho (z) \cdot \langle \mu_z,\varphi \rangle \,dz,
$$
proving \eqref{eq: D_u formula}. Again by \eqref{uniform conv to a vector} we have $T_{n,z}\rightharpoonup \mu_z$, for every $z\in \overline B_1$, and since $T_{n,z}$ enjoys \eqref{est hp Besov p}-\eqref{est hp BV} if $u$ is assumed to satisfy \eqref{Besov_assumpt}-\eqref{BV_assumpt} respectively, the very same estimate transfers to $\mu_z$ by lower semicontinuity of the total variation with respect to the weak convergence of measures. This proves $(i)$. Since $T_{n,0}\equiv 0$ for all $n$, we also get $(ii)$. Finally, $(iii)$ directly follows by the change of variable $z\mapsto -z$, together with the fact that the kernel $\rho$ is chosen to be even.
\end{proof}

\cref{T:ons_crit} shows that, given any scalar-valued  test function $\varphi \in C^0_c(\Omega \times (0,T))$, the dissipation measure in Onsager's critical spaces is characterized as
$$
\langle D[u],\varphi \rangle = \frac14 \int_{B_1} \nabla \rho (z) \cdot V_\varphi(z) \,dz,
$$
for some vector $V_\varphi\in C^0(\overline B_1;\R^d)$, which is moreover $\gamma$-H\"older continuous in $z$, $\gamma$ depending on the Besov regularity of $u$. 
This leads to the following definition. 

\begin{definition}[Local averaged directional flux]
Let $\mu \in C^{0}(\overline B_1; \mathcal{M}_{\rm loc}(\Omega \times (0,T); \R^d ) )$ be any map given by \cref{T:ons_crit} and $\varphi \in C^0_c(\Omega \times (0,T))$. Given  $z\in B_1$, we say that 
$$
 V_\varphi(z) :=\langle \mu_z,\varphi\rangle
$$
is the local energy flux of $u$, averaged on the $\spt \varphi$, in the direction $z$.
\end{definition}

Note that here we are slightly abusing terminology, since $z$, being any element in the unit ball, is not normalized to $|z|=1$. 

Assume that $D[u]$ does not depend on $\rho$. In \cite{DR00} Duchon and Robert investigated the case in which the kernel is chosen to be radial.  
With this choice, our formula \eqref{DRI_formula} gives and additional equivalent expression. Indeed, letting $\rho(z)=\rho_{\rm rad} (|z|)$, by direct computations we deduce
$$
    \langle D[u],\varphi \rangle = -\frac14 \int_0^1 \rho'_{\rm rad} (r) \left( \int_{\partial B_r} V_\varphi \cdot n \, d\mathcal H^{d-1}\right) \,dr,
$$
$n$ being the inward unit normal. Alternatively, one could let 
$$
\rho \rightarrow \frac{1}{\omega_d} \mathds{1}_{B_1} ,\qquad \mathds{1}_{B_1}(z):=\begin{cases}
      1 \quad \text{if } z\in B_1\\
    0 \quad \text{otherwise}
    \end{cases}\,,\qquad \omega_d:=\mathcal H^d(B_1),
$$
to get 
$$
\langle D[u],\varphi \rangle = \frac{1}{4\omega_d}\int_{\partial B_1} V_\varphi \cdot n \,d\mathcal H^{d-1}.
$$

As a byproduct of the formula \eqref{DRI_formula} we get the following result. 

\begin{corollary}[Support restriction]
Denote by $\mathcal U_{x,t}\subset \Omega\times (0,T)$ any open set containing the point $(x,t)$, and assume that $D[u]$ does not depend on $\rho$. It holds
$$
\spt D[u]\subset \left\{ (x,t)\,:\, \forall \mathcal U_{x,t} \, \,\exists \varphi \in C^0_c (\mathcal U_{x,t}) \,\text{ s.t. } \,\inf_{\rho \in \mathcal K}\abs{\int_{B_1}\nabla \rho\cdot V_\varphi }\neq 0 \right\}.
$$
\end{corollary}

It is clear that under any of the assumptions \eqref{Besov_assumpt}-\eqref{BV_assumpt}, since $\{D_\eps [u]\}_\e$ turns out to be a bounded sequence in $L^1_{\rm loc}(\Omega\times (0,T))$, then $D[u]\in \mathcal{M}_{\rm loc}(\Omega\times (0,T))$. Indeed, under the $BV\cap L^\infty$ assumption \eqref{BV_assumpt}, we have that
$$
\int |\delta_{\eps z} u(x,t)|^3\, dx \, dt\lesssim \eps.
$$
However, we remark that $D_\eps[u]$ is bounded in $L^1_{\rm loc}$ provided that
\begin{equation}\label{symmetric_trilin}
\int |\nabla \rho(z)\cdot \delta_{\eps z} u(x,t)||\delta_{\eps z} u(x,t)|^2\, dx \,dt\lesssim \eps.
\end{equation}
If we restrict the kernels $\rho$ to be radial, i.e. 
$$
\rho\in \mathcal K_{\rm rad}:=\left\{\rho\in \mathcal K\, :\, \rho\, \text{ radial} \right\},
$$
an assumption on the symmetric part of $\nabla u$ is enough to deduce \eqref{symmetric_trilin}, as opposite to the one on the full gradient in \eqref{BV_assumpt}. Thus, instead of $BV$, it is enough to consider $BD$, the space of \quotes{bounded deformation} vector fields. However, 
the assumption on the symmetric gradient does not seem enough to get equicontinuity with respect to $z$ of the trilinear expression which is used to generate the measure $\mu_z$. For this reason, in this case, we can only conclude that $D[u]\in \mathcal M_{\rm loc} (\Omega\times (0,T))$. Indeed, building on \eqref{BD_est_increment} from \cref{L:BD_gradient}, we have the following result.

\begin{proposition}[Dissipation measure in $BD$]\label{P:measure_in_BD}
Let $u:\Omega\times (0,T)\rightarrow \R^d$ be a vector field such that 
\begin{equation}\label{BD_assumpt}
u\in L^1(I;BD(O))\cap L^{\infty}(O \times I),
\end{equation}
for all $I\subset\joinrel\subset (0,T)$ and $O\subset\joinrel\subset \Omega$ open sets. Then, $\{D_\eps[u]\}_\e$ is a bounded sequence in $L^1_{\rm loc}(\Omega\times (0,T))$, whenever $\rho\in \mathcal K_{\rm rad}$. In particular $D[u]\in \mathcal M_{\rm loc} (\Omega\times (0,T))$.
\end{proposition}

\begin{proof}

Pick any $\rho(z)=\rho_{\rm rad}(|z|)\in \mathcal K_{\rm rad}$. Then $\nabla \rho (z)=\frac{z}{|z|}\rho'_{\rm rad} (|z|)$, from which, for any  $I\subset\joinrel\subset (0,T)$ and $O\subset\joinrel\subset \Omega$, by  \eqref{BD_est_increment} we get
 \begin{align*}
     \int_{O\times I} |D_\eps [u]|\, dx \, dt&\leq \|u\|^2_{L^\infty((O)_\e \times I)}\int_{B_1}\frac{|\rho'_{\rm rad}|}{|z|}\left(\int_{O\times I} \frac{|z\cdot \delta_{\eps z} u(x,t)|}{\eps}\,dx \, dt \right)\,dz\\
     &\leq C \|u\|^2_{L^\infty((O)_\eps\times I)} \int_I |\nabla^s u(\cdot,t)|\left(\overline{(O)_\eps}\right)\,dt\int_0^1 |\rho'_{\rm rad}(r)| r^{d}\,dr,
 \end{align*}
 for all $2\eps<\dist( O,\partial \Omega)$ and some dimensional constant $C>0$. The right hand side of the previous estimate stays bounded as $\eps\rightarrow 0^+$, and the proof is concluded.

\end{proof}

\section{Interior energy conservation in $BV\cap L^\infty$}\label{S:BV_cons}
Here we prove \cref{T:cons_local_BV}. By \cref{P:Local energy conservation} it is enough to prove that that $D[u]$ satisfies \eqref{cons_cond_general} under the assumption \eqref{BV_assumpt}.

\begin{proof}[Proof of \cref{T:cons_local_BV}]
We prove that $D[u]$ satisfies \eqref{cons_cond_general}. Fix any $O\subset\joinrel\subset K\subset\joinrel\subset \Omega\times (0,T)$, $O$ open, $K$ compact. Pick $\varphi \in C^\infty_c(O)$ and set $S:=\spt \varphi$.
Then, for any convolution kernel $\rho \in \mathcal{K}$, we have 
\begin{align*}
    \abs{\langle D[u] , \varphi \rangle } & = \lim_{\e \to 0} \abs{ \int_S  \left(\int_{B_1} \nabla \rho(z) \cdot \frac{\delta_{\e z} u (x,t)}{4 \e} \abs{\delta_{\e z} u(x,t)}^2 \, dz \right)\varphi(x,t) \, dx \, dt }
    \\ & \lesssim \norm{u}_{L^\infty(K) }^2 \norm{\varphi}_{L^\infty(O)} \limsup_{\e \to 0} \int_{B_1}\left( \int_{S}  \abs{\frac{\nabla \rho(z) \cdot \delta_{\e z} u(x,t)}{\e}} \, dx \, dt \right) dz
    \\ & \lesssim \norm{\varphi}_{L^\infty(O)} \int_{B_1}\left( \limsup_{\e \to 0} \int_{S}   \abs{\frac{\nabla \rho(z) \cdot \delta_{\e z} u(x,t)}{\e}} \, dx \, dt \right) dz,
\end{align*}
where in the last inequality we have used Fatou's lemma. For any $z \in B_1$, define the scalar function $v_z:= u\cdot \nabla_z\rho(z)$, $v_z:\Omega\times (0,T)\rightarrow \R$. Note that the distributional derivative of $v_z$ is given by
$$
\partial_{x_i} v_z(\cdot,t)= \partial_{x_i} u^j \partial_{z_j} \rho(z)=M_{ji}(\cdot,t)\partial_{z_j} \rho(z)|\nabla u(\cdot,t)|,
$$
for a.e. $t\in(0,T)$, where in the last equality we have used the Radon--Nikodym theorem to decompose the matrix-valued measure $\nabla u(\cdot, t)$ with respect to its variation $|\nabla u(\cdot, t)|$. Since $\div u=0$ we deduce that
\begin{equation}
    \label{trace free}
    \trace M (\cdot,t)=0\qquad  \text{for }|\nabla u(\cdot,t)|\text{-a.e. } x\in\Omega , \qquad \text{for a.e. } t\in (0,T). 
\end{equation}
By \eqref{BV_est_increment} applied on $v_z$ we get 
\begin{align*}
   \limsup_{\e \to 0} \int_{S}   \abs{\frac{\nabla \rho(z) \cdot \delta_{\e z} u(x,t)}{\e}} \, dx \, dt&\leq \int_0^T |z\cdot \nabla_x v_z|(\{ x\,:\, (x,t)\in S\})\,dt\\
   &=\int_S |M(x,t) z \cdot \nabla \rho (z) | \,d\nu,
\end{align*}
where the measure $\nu$ is defined by
$$
\nu(A):=\int_0^T |\nabla u(\cdot,t)|(\{ x\,:\, (x,t)\in A\})\,dt\qquad \forall A\subset\joinrel\subset \Omega\times (0,T).
$$
The vector $\eta_{x,t}(z):=M(x,t) z$ is divergence-free for $\nu$-a.e. $(x,t)$ by \eqref{trace free} and, together with $\nu$, it satisfies the assumptions of \cref{P:Local energy conservation} with $m=d+1$ and $U=\Omega\times (0,T)$.
\end{proof}

We remark that if $u\in L^1_t BD_x\cap L^\infty_{x,t}$, by using \eqref{BD_est_increment}, one can prove an estimate similar to \eqref{cons_cond_general} restricting to radial kernels. However, considering $\rho\in \mathcal K_{\rm rad}$ prevents the use of the Alberti's optimization of \cref{L:alberti}, being the latter anisotropic. Details are left to the reader.

\begin{remark}[Necessity of the Duchon--Robert formula]\label{R:DR_vs_CET}
The formula \eqref{DR_eps_formula} proposed by Duchon and Robert in \cite{DR00} is not the only way to approximate the distribution $D[u]$. Indeed, following the original Constantin, E and Titi approach, it has been proved in \cite{I17} that also $R_\eps :\nabla u_\eps$ has the same distributional limit $D[u]$, being $R_\eps =u_\eps \otimes u_\eps -(u\otimes u)_\eps$ the quadratic Reynolds stress tensor arising from mollifying in space the equation at scale $\eps$. However, from such approximation does not seem possible to achieve the bound \eqref{ambrosio_structure}, which then prevents the optimization of the kernel based on the incompressibility.
\end{remark}

\section{Energy conservation on Lipschitz bounded domains}\label{S:boundary}

This section aims to prove global energy conservation on bounded domains $\Omega\subset \R^d$, with Lipschitz boundary, that is \cref{T:energy cons bounded dom intro} together with its more general version \cref{T:en cons bounded domain more general}. To begin, we clarify on the precise meaning of kinetic energy conservation in this setting.
\begin{definition}[Kinetic energy conservation]\label{D:kinetic energy cons}
Given $u:\Omega\times (0,T)\rightarrow \R^d$ such that $u\in L^2(\Omega\times (0,T))$ we define its kinetic energy as
$$
e_u(t):=\frac{1}{2}\int_\Omega |u(x,t)|^2\, dx.
$$
Moreover, we say that the kinetic energy of $u$ is conserved if
$$
\langle e_u,\phi'\rangle=0 \qquad \forall \phi \in C^\infty_c(0,T).
$$
\end{definition}
It is clear that if $u\in C^0([0,T];L^2(\Omega))$, then  kinetic energy conservation is equivalent to $e_u(t)=e_u(0)$ for all $t\in[0,T]$.

\subsection{Normal Lebesgue boundary trace}
 In the spirit of traces for $BV$ function given in \cref{T:trace_in_BV}, we now introduce a notion of trace for the normal component of a vector field $u:\Omega\rightarrow\R^d$, that we call \quotes{normal Lebesgue trace}. Such quantity will play a major role in the boundary analysis.
Recall our notation
$$
d_A(x):=\dist (x,A).
$$

\begin{definition}[Normal Lebesgue boundary trace]\label{D:Leb normal trace}
     Let $\Omega\subset \R^d$ be a Lipschitz domain and $u\in L^1(\Omega;\R^d)$ a vector field. We say that $u$ has a Lebesgue normal trace on $\partial \Omega$ if there exists a function $u_n^{\partial \Omega}\in L^1(\partial \Omega;\mathcal H^{d-1})$ such that, for every sequence $r_k\rightarrow 0^+$, it holds
    $$
    \lim_{k\rightarrow \infty}\frac{1}{r_k^d} \int_{B_{r_k} (x)\cap\Omega} \left|(u\cdot \nabla d_{\partial \Omega})(y)- u_n^{\partial \Omega}(x)\right|\,dy=0\qquad \text{for } \mathcal{H}^{d-1}\text{-a.e. } x\in \partial \Omega.
    $$
\end{definition} 

The choice of taking sequence $r_k\rightarrow 0^+$ instead of continuous radii $r\rightarrow 0^+$ as in \cref{T:trace_in_BV} is made for convenience. It is easy to see that \cref{D:Leb normal trace} is a slightly weaker notion with respect to that using continuous radii, i.e. 
\begin{equation}\label{nomal_trace_cont_radii}
\lim_{r\rightarrow 0^+}\frac{1}{r^d} \int_{B_r (x)\cap\Omega} \left|(u\cdot \nabla d_{\partial \Omega})(y)- u_n^{\partial \Omega}(x)\right|\,dy=0\qquad \text{for } \mathcal{H}^{d-1}\text{-a.e. } x\in \partial \Omega.
\end{equation}
Indeed, letting
$$
f_r(x):=\frac{1}{r^d} \int_{B_r (x)\cap\Omega} \left|(u\cdot \nabla d_{\partial \Omega})(y)- u_n^{\partial \Omega}(x)\right|\,dy,
$$
\eqref{nomal_trace_cont_radii} implies that 
\begin{align*}
\mathcal H^{d-1}\left( \left\{ x\in \partial\Omega\,:\, \lim_{r\rightarrow 0^+} f_r (x)=0\right\}\right) & = \mathcal H^{d-1}\left( \left\{ x\in \partial\Omega\,:\, \forall r_k\rightarrow 0^+\,\lim_{k\rightarrow \infty} f_{r_k} (x)=0\right\}\right)
\\ & =\mathcal H^{d-1} (\partial \Omega).  
\end{align*}
But clearly, if $r_j\rightarrow 0^+$ is any sequence, we have
$$
\left\{ x\in \partial\Omega\,:\, \forall r_k\rightarrow 0^+\,\lim_{k\rightarrow \infty} f_{r_k} (x)=0\right\}\subset \left\{ x\in \partial\Omega\,:\, \lim_{j\rightarrow \infty} f_{r_j} (x)=0\right\},
$$
which implies 
$$
\mathcal H^{d-1} \left( \left\{ x\in \partial\Omega\,:\, \lim_{j\rightarrow \infty} f_{r_j} (x)=0\right\}\right)=\mathcal H^{d-1} (\partial \Omega)\qquad \forall r_j\rightarrow 0^+.
$$
Whenever it exists, the normal Lebesgue trace is unique. Indeed, since $\Omega$ has Lipschitz boundary, see for instance \cite{AFP00}*{Theorem 3.61}, we have 
$$
\lim_{r\rightarrow 0^+}\frac{\mathcal{H}^{d}(\Omega\cap B_r(x))}{\mathcal{H}^{d}(B_r(x))}=\frac{1}{2}\qquad \text{for } \mathcal{H}^{d-1}\text{-a.e. } x\in \partial \Omega.
$$
Thus, if $(u_n^{\partial \Omega})_1,(u_n^{\partial \Omega})_2\in L^1(\partial \Omega;\mathcal H^{d-1})$ are two Lebesgue normal traces of $u$ in the sense of \cref{D:Leb normal trace}, then, for $\mathcal{H}^{d-1}$-a.e.  $x\in \partial \Omega$, we have
\begin{align*}
    \frac{\left|(u_n^{\partial \Omega})_1(x)-(u_n^{\partial \Omega})_2(x)\right|}{\omega_d^{-1}2}&=\lim_{k\rightarrow \infty}\frac{1}{r_k^d}\int_{\Omega\cap B_{r_k}(x)}\left|(u_n^{\partial \Omega})_1(x)-(u_n^{\partial \Omega})_2(x)\right|\,dy\\
    &\leq\lim_{k\rightarrow \infty}\frac{1}{r_k^d}\int_{\Omega\cap B_{r_k}(x)}\left|(u\cdot \nabla d_{\partial\Omega})(y)-(u_n^{\partial \Omega})_1(x)\right|\,dy\\
    &+\lim_{k\rightarrow \infty}\frac{1}{r_k^d}\int_{\Omega\cap B_{r_k}(x)}\left|(u\cdot \nabla d_{\partial\Omega})(y)-(u_n^{\partial \Omega})_2(x)\right|\,dy\\
    &=0.
\end{align*}
In the computations above we have used that $\mathcal H^d(B_r(x))=\omega_d r^d$.

We refer the interested reader to the works \cites{CT05,CTZ07,PT08,CTZ09,CT11,CF99} which build a notion of normal trace for \quotes{measure divergence} vector fields, together with several applications to systems of conservation laws, in setting with very low regularity.

In the next proposition we prove that having zero normal Lebesgue trace implies that the term appearing in the proof of the energy conservation on bounded domains vanishes. In fact, modulo considering subsequences, the two properties are substantially equivalent.
\begin{proposition}\label{P:normal trace absol converg}
Suppose $u\in L^\infty(\Omega;\R^d)$ has zero Lebesgue normal boundary trace, that is  $u_n^{\partial \Omega}\equiv 0$. Then 
\begin{equation}\label{vanish_boundary_flux}
\lim_{\eps \rightarrow 0^+}\int_\Omega |u\cdot \nabla \varphi_\eps|\,dx=0,
\end{equation}
where 
\begin{equation} \label{eq:cut_off}
\varphi_\eps (x):=\left\{\begin{array}{l}
 1 \quad  \quad \quad \quad \quad \,\text{if } x\in \Omega^{\eps}\\
\eps^{-1} d_{\partial \Omega}(x)\quad \text{if } x\in \Omega \setminus \Omega^{\eps}
\end{array}\right.\quad \text{ with } \quad \Omega^\eps :=\left\{x\in \Omega\, :\, d_{\partial \Omega}(x)>\eps \right\}.
\end{equation}
Moreover, if \eqref{vanish_boundary_flux} is true, then, for every sequence $r_k\rightarrow 0^+$, it must hold
\begin{equation}
    \label{normale_trace_liminf}
    \liminf_{k\rightarrow \infty}\frac{1}{r_k^d} \int_{B_{r_k} (x)\cap\Omega} \left|(u\cdot \nabla d_{\partial \Omega})(y)\right|\,dy=0\qquad \text{for } \mathcal{H}^{d-1}\text{-a.e. } x\in \partial \Omega.
\end{equation}
\end{proposition}

\begin{remark}\label{R:trace is bounded}
We remark that as soon as $u\in L^\infty(\Omega;\R^d)$ and it admits a Lebesgue normal trace in the sense of \cref{D:Leb normal trace}, then necessarily $u_n^{\partial \Omega}\in L^\infty(\partial \Omega;\mathcal H^{d-1})$. Indeed, $u_n^{\partial \Omega}$ is the point-wise limit, as $k\rightarrow \infty$, of the sequence
$$
h_k(x):=\frac{1}{r_k^d}\int_{\Omega\cap B_{r_k}(x)} (u\cdot \nabla d_{\partial \Omega})(y) \, dy,
$$
which is bounded in $L^\infty(\partial \Omega;\mathcal{H}^{d-1})$ if $u\in L^\infty(\Omega;\R^d)$.
\end{remark}
\begin{proof}
We prove the first implication. We fix a sequence $\eps_k\rightarrow 0^+$. We have to show that
\begin{equation}\label{claim 1}
   \lim_{k\rightarrow \infty} I_{\eps_k}:= \lim_{k\rightarrow \infty}\frac{1}{\eps_k} \int_{\Omega\setminus \Omega^{\eps_k}} \left|(u\cdot \nabla d_{\partial \Omega})(y)\right|\,dy=0.
    \end{equation}
Define the sequence of functions $f_{\eps_k}:\partial \Omega\rightarrow\R$ as
$$
f_{\eps_k}(x):=\frac{1}{\eps_k^d} \int_{B_{\eps_k}(x)\cap\Omega} \left|(u\cdot \nabla d_{\partial \Omega})(y)\right|\,dy.
$$
By assumption, $f_{5\eps_k}(x)\rightarrow 0$ for $\mathcal H^{d-1}$-a.e. $x\in \partial\Omega$. Fix any $\delta>0$. By Egorov's theorem we find $A_\delta \subset \partial \Omega$, closed in the induced topology on $\partial\Omega$, such that $\mathcal{H}^{d-1}(\partial \Omega\setminus A_\delta)<\delta$ and $f_{5\eps_k}\rightarrow 0$ uniformly on $A_\delta$. By the uniform convergence on $A_\delta$, we can find $k_0=k_0(\delta)>0$ such that, for all $k\geq k_0$  it holds
\begin{equation}
    \label{unif conv Adelta}
    \sup_{\substack{  x\in A_\delta }} f_{5\eps_k}(x)\leq \delta.
\end{equation}
For any fixed $k\geq k_0$ consider the family of balls $\left\{ B_{\eps_k} (x)\right\}_{x\in A_\delta}$. By Vitali's covering theorem we find a finite and disjoint sub-family of balls $\mathcal F:=\{B_{\eps_k} (x_j)\}_{j\in J_{\eps_k}}$,  such that 
\begin{equation*}
   A_\delta\subset (A_\delta)_{\eps_k}\subset  \bigcup_{j\in J_{\eps_k} }B_{5\eps_k} (x_j).
    \end{equation*}
 Here, and throughout this whole section, we will use the notation $(A)_\eps$ to denote the $\eps$-tubular neighbourhood of a set $A$, as already introduced in \cref{S:tools}.
    Since $\mathcal F$ is a disjoint family and $\partial\Omega$ has Minkowski dimension $d-1$, by possibly choosing $k_0$ even larger, it must hold 
    \begin{equation}
        \label{cardinality bounded}
        \#(\mathcal F)\leq c \eps_k^{1-d},
    \end{equation}
    for some constant $c>0$ which depends only  on $\Omega$.
Thus we split
$$
I_{\eps_k}=\frac{1}{\eps_k} \int_{(A_\delta)_{\eps_k}\cap \Omega} \left|(u\cdot \nabla d_{\partial \Omega})(y)\right|\,dy+\frac{1}{\eps_k} \int_{(\Omega\setminus\Omega^{\eps_k})\setminus(A_\delta)_{\eps_k}} \left|(u\cdot \nabla d_{\partial \Omega})(y)\right|\,dy=:I^1_{\eps_k}+I^2_{\eps_k}.
$$
By \eqref{unif conv Adelta} and \eqref{cardinality bounded} we have 
$$
I^1_{\eps_k}\leq \frac{1}{\eps_k}\sum_{j\in J_{\eps_k}} \int_{B_{5\eps_k}(x_j)\cap \Omega} \left|(u\cdot \nabla d_{\partial \Omega})(y)\right|\,dy=5^d\eps_k^{d-1} \sum_{j\in J_{\eps_k}} f_{5\eps_k} (x_j)\leq C\delta.
$$
Moreover, since $u\in L^\infty(\Omega;\R^d)$, we can bound
$$
I^2_{\eps_k} \leq C \frac{1}{\eps_k}\mathcal H^{d}\left( (\Omega\setminus\Omega^{\eps_k})\setminus(A_\delta)_{\eps_k}\right)\leq C \frac{1}{\eps_k}\mathcal H^{d}\left( (\partial \Omega)_{\eps_k}\setminus(A_\delta)_{\eps_k}\right).
$$
Since both $\partial \Omega, A_\delta\subset \R^d$  are closed $(d-1)$-rectifiable sets, by  \eqref{minkow=hausd} from \cref{P:federer minkowski rectifiable} we have 
$$
\lim_{\eps\rightarrow 0^+}\frac{\mathcal H^{d}((\partial \Omega)_\eps)}{\eps}= c \mathcal{H}^{d-1} (\partial \Omega)\quad \text{and} \quad \lim_{\eps\rightarrow 0^+}\frac{\mathcal H^{d}((A_\delta)_\eps)}{\eps}= c \mathcal{H}^{d-1} (A_\delta)
$$
for some purely dimensional constant $c>0$. Thus, since $(\partial\Omega)_\eps=(A_\delta)_\eps\cup \left((\partial\Omega)_\eps\setminus (A_\delta)_\eps\right)$, we deduce
$$
\lim_{\eps\rightarrow 0^+}\frac{\mathcal H^{d}((\partial \Omega)_\eps\setminus(A_\delta)_{\eps})}{\eps}=\lim_{\eps\rightarrow 0^+}\frac{\mathcal H^{d}((\partial \Omega)_\eps)}{\eps}-\lim_{\eps\rightarrow 0^+}\frac{\mathcal H^{d}((A_\delta)_{\eps})}{\eps}=c\mathcal H^{d-1}(\partial \Omega\setminus A_\delta).
$$
Since $\mathcal{H}^{d-1} (\partial \Omega\setminus A_\delta)<\delta$, by possibly choosing $k_0$ larger, we deduce that $I_{\eps_k}^2\leq C \delta$, for all $k\geq k_0$. Summing up, we have proved that for any $\delta>0$, we can find a $k_0\in \N$ large enough such that $I_{\eps_k}\leq C\delta$ for all $k\geq k_0$, for some constant $C>0$ depending only on the domain $\Omega$. This gives \eqref{claim 1} and concludes the proof of the first part.

Let us now prove the second implication by contradiction. If \eqref{normale_trace_liminf} does not hold, then we find a sequence of radii $r_k\rightarrow 0^+$ such that 
\begin{equation}\label{liminf_with_a_seq}
\liminf_{k\rightarrow \infty} f_k (x):= \liminf_{k\rightarrow \infty}\frac{1}{r_k^d} \int_{B_{r_k} (x)\cap\Omega} \left|(u\cdot \nabla d_{\partial \Omega})(y)\right|\,dy>0
\end{equation}
on a set $C\subset \partial \Omega$ with $\mathcal H^{d-1} (C)>0$. We have $C=\bigcup_{m=1}^\infty C_m$, where
$$
C_m:=\left\{ x\,:\, \liminf_{k\rightarrow \infty} f_k(x)\geq \frac{1}{m}\right\}.
$$
Since $C_m \nearrow C$, we find $m$ large enough so that 
$$
\mathcal H^{d-1} (C_m)\geq \frac{\mathcal H^{d-1}(C)}{2}>0.
$$
Apply \cref{L:egorov_one_side} with $(X,\mu)=(C_m,\mathcal{H}^{d-1})$ and $L=\frac{1}{m}$, to find $\tilde C_m\subset C_m$ with 
$$
\mathcal H^{d-1} (\tilde C_m)\geq \frac{\mathcal H^{d-1}(C)}{4}>0
$$
and 
\begin{equation}
    \label{liminf_bounded_below_unif}
    \liminf_{k\rightarrow \infty} \inf_{x\in \tilde C_m} f_k(x)\geq \frac{1}{m}.
\end{equation}
Moreover, by interior approximation of Radon measures, we can also find a set $M\subset \tilde C_m$, closed in the induced topology on $\partial \Omega$, such that
$$
\mathcal H^{d-1} (M)\geq \frac{\mathcal H^{d-1}(C)}{8}>0.
$$
Then, cover $(M)_\e$ with $\{B_\eps(x)\}_{x\in M}$ and extract a finite and disjoint Vitali subcovering such that 
\begin{equation}
    \label{subcovering_vitali}
    (M)_\eps \subset \bigcup_{i=1}^{N_\eps} B_{5\eps}(x_i).
\end{equation}
From \eqref{subcovering_vitali} and \eqref{minkow=hausd}, if $\eps$ is sufficiently small, we obtain
$$
\frac{\mathcal H^{d-1} (C)}{16}\leq \frac{\mathcal H^{d-1} (M)}{2}\leq \frac{\mathcal H^{d} ((M)_\eps)}{\omega_1 \eps}\leq 5^d \frac{\omega_d}{\omega_1} N_\eps \eps^{d-1},
$$
from which
\begin{equation}
    \label{liminf_lower_bound_N_eps}
    \liminf_{\eps\rightarrow 0^+} N_{\eps} \eps^{d-1} \geq c \frac{\mathcal H^{d-1} (C)}{16}>0,
\end{equation}
for some dimensional constant $c>0$. But then, if $r_k$ is the original sequence we have chosen in \eqref{liminf_with_a_seq}, we can bound from below
\begin{align*}
    \int_{\Omega} \left| u\cdot \nabla \varphi_{r_k}\right| \, dy &\geq \frac{1}{r_k}\int_{(M)_{r_k} \cap \Omega}|u\cdot \nabla d_{\partial \Omega}| \, dy \geq \frac{1}{r_k}\int_{\bigcup_{i=1}^{N_{r_k}} B_{r_k}(x_i) \cap \Omega}|u\cdot \nabla d_{\partial \Omega}| \, dy \\
    &=\sum_{i=1}^{N_{r_k}} \frac{1}{r_k}\int_{ B_{r_k}(x_i) \cap \Omega}|u\cdot \nabla d_{\partial \Omega}| \, dy =\sum_{i=1}^{N_{r_k}} r_k^{d-1} f_k (x_i)\\
    &\geq \sum_{i=1}^{N_{r_k}} r_k^{d-1} \inf_{i} f_k (x_i)\geq N_{r_k} r_k^{d-1} \inf_{x\in M} f_{k} (x).
\end{align*}
Letting $k\rightarrow \infty$, by \eqref{liminf_bounded_below_unif} and \eqref{liminf_lower_bound_N_eps} we deduce
$$
\liminf_{k\rightarrow \infty} \int_{\Omega} \left| u\cdot \nabla \varphi_{r_k}\right| \, dy \geq c\frac{\mathcal{ H}^{d-1}(C)}{m16}>0,
$$
which shows that \eqref{vanish_boundary_flux} cannot hold.
\end{proof}

Recall that solutions of Euler, say $u\in L^2(\Omega\times (0,T))$, on a Lipschitz bounded domain $\Omega$ satisfy 
$$
\int_\Omega u(x,t)\cdot \nabla \varphi(x)\,dx= 0 \qquad \forall\varphi \in H^1(\Omega),
$$
for a.e. $t\in (0,T)$. The latter condition is indeed a way to define, in a weak sense, $L^2$ divergence-free vector fields tangent to $\partial \Omega$. However, this does not automatically imply that the Lebesgue normal boundary trace of $u$, in the sense of \cref{D:Leb normal trace}, vanishes. In fact, it is not even clear in general whether such a notion of trace exists. We give some conditions which make this true. 

\begin{proposition}\label{P:normal trace is zero}
Let $\Omega\subset \R^d$ be a bounded Lipschitz domain and let $u\in L^\infty(\Omega;\R^d)$ be a divergence-free vector field, tangent to the boundary, that is 
\begin{equation} \label{eq:tangent_div_free}
\int_\Omega u(x)\cdot \nabla \varphi(x)\,dx= 0 \qquad \forall\varphi \in W^{1,1}(\Omega).  
\end{equation}

Then $u_n^{\partial \Omega}\equiv 0$ if one of the following conditions holds:
\begin{itemize}
    \item[(i)] For $\mathcal H^{d-1}$-a.e. $x\in \partial \Omega$, $\exists\eps_0>0$ such that $u\cdot \nabla d_{\partial \Omega}\in C^0\left(B_{\eps_0}(x)\cap \overline \Omega\right)$;
    \item[(ii)] $u\in BV(\Omega;\R^d)$.
\end{itemize}
\end{proposition}
In particular, we emphasise that the assumption $(i)$ relaxes the usual one in which $u\cdot \nabla d_{\partial \Omega}$ is continuous in a full open neighbourhood of $\partial \Omega$. For instance, when the domain is piece-wise $C^2$, this allows us to deduce kinetic energy conservation on Lipschitz domains, without imposing anything in neighbourhoods of corner points of $\partial \Omega$.

\begin{proof}
Given $\e>0$, in the same notation of \eqref{eq:cut_off}, we define 
\begin{equation}
    \varphi_\e(x) = 
    \begin{cases}
        \e^{-1} d_{\partial \Omega} (x) & x \in \Omega \setminus \Omega^\e, 
        \\ 1 & x \in  \Omega^\e. 
    \end{cases}
\end{equation}
Then, by \eqref{eq:tangent_div_free}, for any test function $g \in C^1(\R^d)$,  we have that 
\begin{align*}
    \int_{\Omega} (u \cdot \nabla \varphi_\e) g \, dx & = \int_{\Omega} u \cdot \nabla (\varphi_\e  g) \, dx - \int_{\Omega} (u \cdot \nabla g)\varphi_\e \, dx  = - \int_{\Omega} (u\cdot \nabla g) \varphi_\e  \, dx. 
\end{align*}
For $\e \to 0$,  $\varphi_\e \to \mathds{1}_{\Omega}$ in $L^1(\Omega)$, from which
\begin{equation}\label{conv to zero}
    \lim_{\e \to 0} \int_{\Omega} (u\cdot  \nabla \varphi_\e) g\, dx = - \int_{\Omega} u \cdot\nabla g \, dx = 0. 
\end{equation}
Take $x\in \partial \Omega$ such that $u\cdot \nabla d_{\partial \Omega}\in C^0(B_{\eps_0}(x)\cap \overline \Omega)$, for some $\eps_0>0$. It is clear that if 
\begin{equation}
    \label{normal_comp_vanish_boundary}
    u\cdot \nabla d_{\partial \Omega}\Big|_{B_{\eps_0}(x)\cap \partial \Omega} \equiv 0,
\end{equation}
then, by the uniform continuity of $u\cdot \nabla d_{\partial \Omega}$ in $B_r(x)\cap \overline \Omega$, for all $r<\eps_0$, we get that 
$$
\lim_{r\rightarrow 0^+}\frac{1}{r^d}\int_{B_r(x)\cap \Omega} \left| (u\cdot \nabla d_{\partial \Omega} )(y)\right|\,dy=0,
$$
that is, the Lebesgue normal boundary trace of $u$ is zero. Thus we are left to prove \eqref{normal_comp_vanish_boundary}. By contradiction, suppose that $\exists z\in B_{\eps_0}(x)\cap \partial \Omega$ such that $(u\cdot \nabla d_{\partial \Omega})(z)=c\neq 0$. To fix the ideas, assume $c>0$. The proof in the case $c<0$ follows by straightforward modifications. Since $u\cdot \nabla d_{\partial \Omega}\in C^0(B_{\eps_0}(x)\cap \overline \Omega)$, we can find an $\eps_1=\eps_1(c)>0$ such that $B_{2\eps_1}(z)\subset \joinrel \subset B_{\eps_0}(x)$ and 
$$
u\cdot \nabla d_{\partial \Omega}\Big|_{B_{2\eps_1}(z)\cap \overline \Omega}\geq \frac{c}{2}.
$$
Pick $g\in C^1_c(B_{2\eps_1}(z))$ such that $g\geq 0$ and $g\big|_{B_{\eps_1}(z)}\equiv 1$. Then, for $\eps>0$ sufficiently small, we get
\begin{align*}
\int_\Omega u\cdot \nabla \varphi_\eps g\,dy &\geq \frac{1}{\eps}\int_{B_{\eps_1} (z)\cap (\Omega \setminus \Omega^\eps)}u\cdot \nabla d_{\partial \Omega}\,dy\\
&\geq \frac{c}{2}\frac{\mathcal{H}^d(B_{\eps_1} (z)\cap (\Omega \setminus \Omega^\eps))}{\eps}\\
&\geq C \mathcal{H}^{d-1}\left(\overline B_{\sfrac{\eps_1}{2}}(z)\cap \partial \Omega \right)>0,
\end{align*}
where in the last inequality we have applied \eqref{lower bound half density point} with $C=\overline B_{\sfrac{\eps_1}{2}}(z)\cap \partial \Omega$. This is in contradiction with \eqref{conv to zero}, and concludes the proof of $(i)$.

Assume that $u \in BV(\Omega;\R^d)$. Letting $\Tilde{u}$ be the zero extension of $u$ outside $\Omega$, by \cref{T:trace_in_BV}, we get that $\Tilde{u} \in BV(\R^d;\R^d)$ and the distributional gradient is given by 
$$\nabla  \Tilde{u} = (\nabla  u) \llcorner\Omega+ (u^\Omega \otimes n  ) \mathcal{H}^{d-1} \llcorner \partial \Omega \in \mathcal M(\R^d;\R^{d\times d}), $$
where $n$ is the inward unit normal to $\partial\Omega$. Note that \eqref{eq:tangent_div_free} is equivalent to the fact that $\Tilde{u}$ is divergence-free in $\R^d$. Hence, we have that 
$$0 = \div \Tilde{u} = (\div u) \llcorner \Omega +  (u^\Omega \cdot n)\mathcal{H}^{d-1}\llcorner \partial \Omega =  (u^\Omega\cdot n)  \mathcal{H}^{d-1}\llcorner \partial \Omega, $$
yielding $u^\Omega(x)\cdot n(x)  = 0$ for $\mathcal{H}^{d-1}$-a.e. $x \in \partial \Omega$. It remains to check that $u^\Omega\cdot n$ is the Lebesgue normal trace of $u$ on $\partial \Omega$ in the sense of \cref{D:Leb normal trace}. Indeed, for $\mathcal{H}^{d-1}$-a.e. $x \in \partial \Omega$, we have that 
\begin{align*}
    \frac{1}{r_k^d} \int_{B_{r_k}(x) \cap \Omega} \abs{(u \cdot\nabla d_{\partial \Omega})(y) - (u^{\Omega}\cdot n)(x) }\, dy & \leq \frac{1}{r_k^d} \int_{B_{r_k}(x) \cap \Omega} \abs{u(y) - u^{\Omega}(x)} \abs{\nabla d_{\partial \Omega}(y)}\, dy 
    \\ & + \frac{\abs{u^\Omega(x)}}{r_k^d} \int_{B_{r_k}(x) \cap \Omega} \abs{\nabla d_{\partial \Omega}(y) - n (x)}\, dy. 
\end{align*}
The first term goes to $0$ as $r_k \to 0$, by the definition of the trace for $BV$ functions (see \cref{T:trace_in_BV}) and the fact that $\abs{\nabla d_{\partial \Omega}(y)} \leq 1$ a.e. in $\Omega$. Similarly to \cref{R:trace is bounded}, from $u\in L^\infty(\Omega;\R^d)$ we deduce $u^\Omega\in L^\infty(\partial \Omega;\mathcal H^{d-1})$. Thus, \cref{T:dist function and normal} concludes the proof.

\end{proof}

\subsection{Kinetic energy conservation on bounded domains}

We are ready to prove energy conservation on bounded domains. 

\begin{proof}[Proof of \cref{T:energy cons bounded dom intro}]
Being $u$ a solution of \eqref{E}, we get that $D[u]$ does not depend on the choice of the kernel $\rho$. Thus, by \cref{T:cons_local_BV} we get $|D[u]|(K)=0$ for every $K\subset\joinrel\subset \Omega\times (0,T)$ compact. In particular, recalling that $u\in L^\infty_{\rm loc}$ and $p\in L^1_{\rm loc}$ by assumption, the following identity holds in $\mathcal D'(\Omega\times (0,T))$
\begin{equation}\label{local_energy_eq_used}
\partial_t  \frac{\abs{u}^2}{2}  + \div \left(  \left( \frac{\abs{u}^2}{2} + p \right)u \right) = 0. 
\end{equation}
Fix any $\phi \in C^1_c((0,T))$ and let $I\subset \joinrel \subset (0,T)$ be any time interval such that $\spt \phi\subset I$. Then, take the corresponding $\eps_0>0$ as in the assumption $(ii)$. For every  $\varphi \in C^\infty_c(\Omega)$ such that $\varphi \big|_{\Omega^{\eps_0}} \equiv 1$ (here we are using the notation $\Omega^{\eps_0}$ introduced in \eqref{eq:cut_off}) we have that 
\begin{align*}
    \abs{\int_0^T \phi'(t) \int_{\Omega} \frac{\abs{u}^2}{2} \varphi(x) \, dx \, dt} & = \abs{\int_0^T \phi(t) \int_{\Omega \setminus \Omega^{\eps_0}} \left( \frac{\abs{u}^2}{2} + p\right) u \cdot\nabla \varphi \, dx \,  dt }
    \\ & \leq C \int_I f(t)\int_{\Omega} \abs{u (x,t) \cdot \nabla \varphi(x)}\, dx \, dt,
\end{align*}
where 
$$
f(t):=\|u(t)\|^2_{L^\infty(\Omega)} +\|p(t)\|_{L^\infty((\partial \Omega)_{\eps_0}\cap \Omega)}\in L^1(I).
$$
Since $u\in L^\infty(I;L^\infty(\Omega))$, by a standard density argument the previous inequality remains true for all $\varphi\in W^{1,1}_0(\Omega)$. Thus, for every $0<\eps<\eps_0(I)$, we can choose $\varphi=\varphi_\eps$ as in \eqref{eq:cut_off}.
Letting $\e \to 0^+$ and recalling that $\varphi_\e \to \mathds{1}_{\Omega}$ in $L^1(\Omega)$, by Fatou's lemma and \cref{P:normal trace absol converg}, we have that
$$\abs{ \int_0^T \phi'(t) e_u(t) \, dt } \leq C \liminf_{\eps\rightarrow 0^+}\int_0^T f(t) \int_{\Omega} \abs{u (x,t) \cdot \nabla \varphi_\eps(x)}\, dx \, dt=0.  $$
In other words, since $\phi\in C^1_c((0,T))$ was arbitrary, the total kinetic energy is conserved in the sense of \cref{D:kinetic energy cons}. 
\end{proof}

Although \cref{T:energy cons bounded dom intro} provides kinetic energy conservation under quite explicit assumptions, such as the conditions given in \cref{P:normal trace is zero}, it is clear from the proof that the global $L^\infty$ assumption on $u$ and $p$ can be relaxed. Indeed, in the spirit of \cite{BTW19}, we have the following result.

\begin{theorem}
    \label{T:en cons bounded domain more general}
    Let $u\in L^2(\Omega\times (0,T))$ be a solution to \eqref{E} on a Lipschitz bounded domain $\Omega\subset \R^d$, such that $u\in L^1(I;BV (O))\cap L^\infty(O\times I)$ and $p\in L^1(O\times I)$, for all $I\subset \joinrel\subset (0,T)$ and $O\subset\joinrel\subset \Omega$ open sets. Suppose that
    \begin{equation}\label{cond en cons implicit}
        \liminf_{\eps\rightarrow 0^+}\left|\int_I \int_{\Omega} \left(\frac{|u|^2}{2}+p \right) u\cdot \nabla \varphi_\eps \,dx \, dt\right|=0 \qquad \forall I\subset \joinrel\subset (0,T),
    \end{equation}
    where 
    \begin{equation*}
\varphi_\eps (x):=\left\{\begin{array}{l}
 1 \quad  \quad \quad \quad \quad \,\text{if } x\in \Omega^{2\eps}\\
\eps^{-1} d_{\partial \Omega^\eps}(x)\quad \text{if } x\in \Omega^\eps \setminus \Omega^{2\eps}\\
0 \quad  \quad \quad \quad \quad \,\text{if } x\in \Omega\setminus\Omega^{\eps}
\end{array}\right.\quad \text{ with } \quad \Omega^\eps :=\left\{x\in \Omega\, :\, d_{\partial \Omega}(x)>\eps \right\}.
\end{equation*}
    Then, $u$ conserves the kinetic energy in the sense of \cref{D:kinetic energy cons}. 
\end{theorem}
Differently from \cite{BTW19}, our result provides energy conservation in a critical Onsager class. Being the same argument already used above in the proof of \cref{T:energy cons bounded dom intro}, we omit the details.

It is clear that \cref{T:en cons bounded domain more general} assumes a more general condition with respect to \cref{T:energy cons bounded dom intro} where both $u$ and $p$ are assumed to be bounded in a neighbourhood of $\partial \Omega$. Indeed, one can estimate
\begin{align*}
\int_0^T \int_{\Omega} \left|\left(\frac{|u|^2}{2}+p \right) u\cdot \nabla \varphi_\eps \right|\,dx \, dt\leq  \|P\|_{L^\infty((\Omega^\eps\setminus \Omega^{2\eps} ) \times I)}\int_0^T\int_{\Omega} \left| u\cdot \nabla \varphi_\eps \right|\,dx \, dt,
\end{align*}
where $P:=\frac{|u|^2}{2}+p$ denotes the Bernoulli pressure. In particular, $\|P\|_{L^\infty((\Omega^\eps\setminus \Omega^{2\eps}) \times I)}$ is also allowed to blow-up in the limit as $\eps\rightarrow 0^+$, provided that the last term $\int_0^T\int_{\Omega} \left| u\cdot \nabla \varphi_\eps \right|$ compensates. This happens, for instance, if $u$ achieves its Lebesgue normal boundary trace in a fast enough way.

\bibliographystyle{plain} 
\bibliography{biblio}

\begin{bibdiv}
\begin{biblist}

\bib{Al93}{article}{
      author={Alberti, G.},
       title={Rank one property for derivatives of functions with bounded
  variation},
        date={1993},
        ISSN={0308-2105,1473-7124},
     journal={Proc. Roy. Soc. Edinburgh Sect. A},
      volume={123},
      number={2},
       pages={239\ndash 274},
         url={https://doi.org/10.1017/S030821050002566X},
      review={\MR{1215412}},
}

\bib{Ambr04}{article}{
      author={Ambrosio, L.},
       title={Transport equation and {C}auchy problem for {$BV$} vector
  fields},
        date={2004},
        ISSN={0020-9910},
     journal={Invent. Math.},
      volume={158},
      number={2},
       pages={227\ndash 260},
         url={https://doi.org/10.1007/s00222-004-0367-2},
      review={\MR{2096794}},
}

\bib{AFP00}{book}{
      author={Ambrosio, L.},
      author={Fusco, N.},
      author={Pallara, D.},
       title={Functions of bounded variation and free discontinuity problems},
      series={Oxford Mathematical Monographs},
   publisher={The Clarendon Press, Oxford University Press, New York},
        date={2000},
        ISBN={0-19-850245-1},
      review={\MR{1857292}},
}

\bib{BTW19}{article}{
      author={Bardos, C.},
      author={Titi, E.~S.},
      author={Wiedemann, E.},
       title={Onsager's conjecture with physical boundaries and an application
  to the vanishing viscosity limit},
        date={2019},
        ISSN={0010-3616},
     journal={Comm. Math. Phys.},
      volume={370},
      number={1},
       pages={291\ndash 310},
         url={https://doi.org/10.1007/s00220-019-03493-6},
      review={\MR{3982696}},
}

\bib{Ber23}{article}{
      author={Berselli, L.~C.},
       title={Energy conservation for weak solutions of incompressible fluid
  equations: the {H}\"{o}lder case and connections with {O}nsager's
  conjecture},
        date={2023},
        ISSN={0022-0396,1090-2732},
     journal={J. Differential Equations},
      volume={368},
       pages={350\ndash 375},
         url={https://doi.org/10.1016/j.jde.2023.06.002},
      review={\MR{4603827}},
}

\bib{BG23}{article}{
      author={Berselli, L.~C.},
      author={Georgiadis, S.},
       title={Three results on the energy conservation for the 3d {E}uler
  equations},
        date={2023},
        note={Preprint available at \href{https://arxiv.org/abs/2307.04410}{
  arXiv:2307.04410}},
}

\bib{BDSV19}{article}{
      author={Buckmaster, T.},
      author={De~Lellis, C.},
      author={Sz\'{e}kelyhidi, L., Jr.},
      author={Vicol, V.},
       title={Onsager's conjecture for admissible weak solutions},
        date={2019},
        ISSN={0010-3640,1097-0312},
     journal={Comm. Pure Appl. Math.},
      volume={72},
      number={2},
       pages={229\ndash 274},
         url={https://doi.org/10.1002/cpa.21781},
      review={\MR{3896021}},
}

\bib{CF99}{article}{
      author={Chen, G.-Q.},
      author={Frid, H.},
       title={Divergence-measure fields and hyperbolic conservation laws},
        date={1999},
        ISSN={0003-9527,1432-0673},
     journal={Arch. Ration. Mech. Anal.},
      volume={147},
      number={2},
       pages={89\ndash 118},
         url={https://doi.org/10.1007/s002050050146},
      review={\MR{1702637}},
}

\bib{CT05}{article}{
      author={Chen, G.-Q.},
      author={Torres, M.},
       title={Divergence-measure fields, sets of finite perimeter, and
  conservation laws},
        date={2005},
        ISSN={0003-9527},
     journal={Arch. Ration. Mech. Anal.},
      volume={175},
      number={2},
       pages={245\ndash 267},
         url={https://doi.org/10.1007/s00205-004-0346-1},
      review={\MR{2118477}},
}

\bib{CT11}{article}{
      author={Chen, G.-Q.},
      author={Torres, M.},
       title={On the structure of solutions of nonlinear hyperbolic systems of
  conservation laws},
        date={2011},
        ISSN={1534-0392},
     journal={Commun. Pure Appl. Anal.},
      volume={10},
      number={4},
       pages={1011\ndash 1036},
         url={https://doi.org/10.3934/cpaa.2011.10.1011},
      review={\MR{2787432}},
}

\bib{CTZ07}{article}{
      author={Chen, G.-Q.},
      author={Torres, M.},
      author={Ziemer, W.~P.},
       title={Measure-theoretic analysis and nonlinear conservation laws},
        date={2007},
        ISSN={1558-8599},
     journal={Pure Appl. Math. Q.},
      volume={3},
      number={3, Special Issue: In honor of Leon Simon. Part 2},
       pages={841\ndash 879},
         url={https://doi.org/10.4310/PAMQ.2007.v3.n3.a9},
      review={\MR{2351648}},
}

\bib{CTZ09}{article}{
      author={Chen, G.-Q.},
      author={Torres, M.},
      author={Ziemer, W.~P.},
       title={Gauss-{G}reen theorem for weakly differentiable vector fields,
  sets of finite perimeter, and balance laws},
        date={2009},
        ISSN={0010-3640},
     journal={Comm. Pure Appl. Math.},
      volume={62},
      number={2},
       pages={242\ndash 304},
         url={https://doi.org/10.1002/cpa.20262},
      review={\MR{2468610}},
}

\bib{CCFS08}{article}{
      author={Cheskidov, A.},
      author={Constantin, P.},
      author={Friedlander, S.},
      author={Shvydkoy, R.},
       title={Energy conservation and {O}nsager's conjecture for the {E}uler
  equations},
        date={2008},
        ISSN={0951-7715},
     journal={Nonlinearity},
      volume={21},
      number={6},
       pages={1233\ndash 1252},
         url={https://doi.org/10.1088/0951-7715/21/6/005},
      review={\MR{2422377}},
}

\bib{CET94}{article}{
      author={Constantin, P},
      author={E, W.},
      author={Titi, E.~S.},
       title={Onsager's conjecture on the energy conservation for solutions of
  {E}uler's equation},
        date={1994},
        ISSN={0010-3616},
     journal={Comm. Math. Phys.},
      volume={165},
      number={1},
       pages={207\ndash 209},
         url={http://projecteuclid.org/euclid.cmp/1104271041},
      review={\MR{1298949}},
}

\bib{Cr09}{book}{
      author={Crippa, G.},
       title={The flow associated to weakly differentiable vector fields},
      series={Tesi. Scuola Normale Superiore di Pisa (Nuova Series) [Theses of
  Scuola Normale Superiore di Pisa (New Series)]},
   publisher={Edizioni della Normale, Pisa},
        date={2009},
      volume={12},
        ISBN={978-88-7642-340-6; 88-7642-340-6},
      review={\MR{2512801}},
}

\bib{DS13}{article}{
      author={De~Lellis, C.},
      author={Sz\'{e}kelyhidi, L., Jr.},
       title={Dissipative continuous {E}uler flows},
        date={2013},
        ISSN={0020-9910,1432-1297},
     journal={Invent. Math.},
      volume={193},
      number={2},
       pages={377\ndash 407},
         url={https://doi.org/10.1007/s00222-012-0429-9},
      review={\MR{3090182}},
}

\bib{DT22}{article}{
      author={De~Rosa, L.},
      author={Tione, R.},
       title={Sharp energy regularity and typicality results for {H}\"{o}lder
  solutions of incompressible {E}uler equations},
        date={2022},
        ISSN={2157-5045,1948-206X},
     journal={Anal. PDE},
      volume={15},
      number={2},
       pages={405\ndash 428},
         url={https://doi.org/10.2140/apde.2022.15.405},
      review={\MR{4409882}},
}

\bib{DN18}{article}{
      author={Drivas, T.~D.},
      author={Nguyen, H.~Q.},
       title={Onsager's conjecture and anomalous dissipation on domains with
  boundary},
        date={2018},
        ISSN={0036-1410},
     journal={SIAM J. Math. Anal.},
      volume={50},
      number={5},
       pages={4785\ndash 4811},
         url={https://doi.org/10.1137/18M1178864},
      review={\MR{3851045}},
}

\bib{DR00}{article}{
      author={Duchon, J.},
      author={Robert, R.},
       title={Inertial energy dissipation for weak solutions of incompressible
  {E}uler and {N}avier-{S}tokes equations},
        date={2000},
        ISSN={0951-7715},
     journal={Nonlinearity},
      volume={13},
      number={1},
       pages={249\ndash 255},
         url={https://doi.org/10.1088/0951-7715/13/1/312},
      review={\MR{1734632}},
}

\bib{EG15}{book}{
      author={Evans, L.~C.},
      author={Gariepy, R.~F.},
       title={Measure theory and fine properties of functions},
   publisher={Chapman \& Hall/CRC},
        date={2015},
}

\bib{E94}{article}{
      author={Eyink, Gregory~L.},
       title={Energy dissipation without viscosity in ideal hydrodynamics. {I}.
  {F}ourier analysis and local energy transfer},
        date={1994},
        ISSN={0167-2789,1872-8022},
     journal={Phys. D},
      volume={78},
      number={3-4},
       pages={222\ndash 240},
         url={https://doi.org/10.1016/0167-2789(94)90117-1},
      review={\MR{1302409}},
}

\bib{Fed}{book}{
      author={Federer, H.},
       title={Geometric measure theory},
      series={Die Grundlehren der mathematischen Wissenschaften, Band 153},
   publisher={Springer-Verlag New York, Inc., New York},
        date={1969},
      review={\MR{0257325}},
}

\bib{Frisch95}{book}{
      author={Frisch, U.},
       title={Turbulence},
   publisher={Cambridge University Press, Cambridge},
        date={1995},
        ISBN={0-521-45103-5},
        note={The legacy of A. N. Kolmogorov},
      review={\MR{1428905}},
}

\bib{GT}{book}{
      author={Gilbarg, D.},
      author={Trudinger, N.~S.},
       title={Elliptic partial differential equations of second order},
     edition={Second},
      series={Grundlehren der mathematischen Wissenschaften [Fundamental
  Principles of Mathematical Sciences]},
   publisher={Springer-Verlag, Berlin},
        date={1983},
      volume={224},
        ISBN={3-540-13025-X},
         url={https://doi.org/10.1007/978-3-642-61798-0},
      review={\MR{737190}},
}

\bib{I17}{article}{
      author={Isett, P.},
       title={On the endpoint regularity in {O}nsager's conjecture},
        date={2017},
        note={Preprint available at
  \href{https://arxiv.org/pdf/1706.01549.pdf}{arXiv:1706.01549}},
}

\bib{Is18}{article}{
      author={Isett, P.},
       title={A proof of {O}nsager's conjecture},
        date={2018},
        ISSN={0003-486X,1939-8980},
     journal={Ann. of Math. (2)},
      volume={188},
      number={3},
       pages={871\ndash 963},
         url={https://doi.org/10.4007/annals.2018.188.3.4},
      review={\MR{3866888}},
}

\bib{IO17}{article}{
      author={Isett, P.},
      author={Oh, S-J},
       title={On the kinetic energy profile of {H}\"{o}lder continuous {E}uler
  flows},
        date={2017},
        ISSN={0294-1449,1873-1430},
     journal={Ann. Inst. H. Poincar\'{e} C Anal. Non Lin\'{e}aire},
      volume={34},
      number={3},
       pages={711\ndash 730},
         url={https://doi.org/10.1016/j.anihpc.2016.05.002},
      review={\MR{3633742}},
}

\bib{K41}{article}{
      author={Kolmogoroff, A.~N.},
       title={Dissipation of energy in the locally isotropic turbulence},
        date={1941},
     journal={C. R. (Doklady) Acad. Sci. URSS (N.S.)},
      volume={32},
       pages={16\ndash 18},
      review={\MR{0005851}},
}

\bib{L34}{article}{
      author={Leray, J.},
       title={Sur le mouvement d'un liquide visqueux emplissant l'espace},
        date={1934},
        ISSN={0001-5962},
     journal={Acta Math.},
      volume={63},
      number={1},
       pages={193\ndash 248},
         url={https://doi.org/10.1007/BF02547354},
      review={\MR{1555394}},
}

\bib{NV22}{article}{
      author={Novack, M.},
      author={Vicol, V.},
       title={An intermittent {O}nsager theorem},
        date={2023},
        ISSN={0020-9910,1432-1297},
     journal={Invent. Math.},
      volume={233},
      number={1},
       pages={223\ndash 323},
         url={https://doi.org/10.1007/s00222-023-01185-6},
      review={\MR{4601999}},
}

\bib{O49}{article}{
      author={Onsager, L.},
       title={Statistical hydrodynamics},
        date={1949},
     journal={Nuovo Cimento (9)},
      volume={6},
      number={Supplemento, 2 (Convegno Internazionale di Meccanica
  Statistica)},
       pages={279\ndash 287},
}

\bib{PT08}{article}{
      author={Phuc, N.~C.},
      author={Torres, M.},
       title={Characterizations of the existence and removable singularities of
  divergence-measure vector fields},
        date={2008},
        ISSN={0022-2518},
     journal={Indiana Univ. Math. J.},
      volume={57},
      number={4},
       pages={1573\ndash 1597},
         url={https://doi.org/10.1512/iumj.2008.57.3312},
      review={\MR{2440874}},
}

\bib{Sch93}{article}{
      author={Scheffer, V.},
       title={An inviscid flow with compact support in space-time},
        date={1993},
        ISSN={1050-6926,1559-002X},
     journal={J. Geom. Anal.},
      volume={3},
      number={4},
       pages={343\ndash 401},
         url={https://doi.org/10.1007/BF02921318},
      review={\MR{1231007}},
}

\bib{Shn00}{article}{
      author={Shnirelman, A.},
       title={Weak solutions with decreasing energy of incompressible {E}uler
  equations},
        date={2000},
        ISSN={0010-3616,1432-0916},
     journal={Comm. Math. Phys.},
      volume={210},
      number={3},
       pages={541\ndash 603},
         url={https://doi.org/10.1007/s002200050791},
      review={\MR{1777341}},
}

\bib{Shv09}{article}{
      author={Shvydkoy, R.},
       title={On the energy of inviscid singular flows},
        date={2009},
        ISSN={0022-247X,1096-0813},
     journal={J. Math. Anal. Appl.},
      volume={349},
      number={2},
       pages={583\ndash 595},
         url={https://doi.org/10.1016/j.jmaa.2008.09.007},
      review={\MR{2456214}},
}

\end{biblist}
\end{bibdiv}

\end{document}